%% file: arxiv.tex
\documentclass[hidelinks,onefignum,onetabnum]{siamart220329}

\input{SISC_shared}

\ifpdf
\hypersetup{
  pdftitle={Viscous regularization of the MHD equations},
  pdfauthor={Tuan Anh Dao, Lukas Lundgren, and Murtazo Nazarov}
}
\fi

\usepackage{amssymb}
\usepackage{graphicx} 
\usepackage{xspace}
\usepackage{multirow}
\usepackage{bbm}
\usepackage{Macros/mydef}
\usepackage{Macros/remarks}
\usepackage{algorithm}
\usepackage[numbers]{natbib}
\usepackage{graphicx}
\usepackage{subcaption}
\usepackage{enumitem}
\newsiamthm{example}{Example}
\usepackage{multirow}
\usepackage{adjustbox}
\usepackage{amsmath}
\usepackage{changepage}
\usepackage{mathtools}
\usepackage{rotating}
\usepackage{pdflscape}
\usepackage[normalem]{ulem}

\newcommand{\cred}[1]{{\color{black}#1}}
\newcommand{\cblue}[1]{{\color{black}#1}}
\newcommand{\cgreen}[1]{{\color{black}#1}} 
\DeclareMathSymbol{:}{\mathord}{operators}{"3A}

\begin{document}

\maketitle



\begin{abstract}
Nonlinear conservation laws such as the system of ideal magnetohydrodynamics (MHD) equations may develop singularities over time. In these situations, viscous regularization is a common approach to regain regularity of the solution. In this paper, we present a new viscous flux to regularize the MHD equations which holds many attractive properties. In particular, we prove that the proposed viscous flux preserves positivity of density and internal energy, satisfies the minimum entropy principle, is consistent with all generalized entropies, and is Galilean and rotationally invariant. We also provide a variation of the viscous flux that conserves angular momentum. To make the analysis more useful for numerical schemes, the divergence of the magnetic field is not assumed to be zero. Using continuous finite elements, we show several numerical experiments including contact waves and magnetic reconnection.
\end{abstract}

\begin{keywords}
MHD, viscous regularization, artificial viscosity, entropy principles
\end{keywords}

\begin{MSCcodes}
34A45
\end{MSCcodes}

\section{Introduction}\label{Sec:Introduction}
Dynamics of conducting fluids, such as plasma, when viscosity and magnetic resistivity are neglected can be numerically studied by solving the equations of ideal MHD. Besides accuracy and stability, preserving physical features of numerical solutions is very desirable and sometimes is even essential. For instance, the laws of physics break down when fluid density or internal energy becomes negative. Many difficulties in avoiding such situations come from the complex nonlinearity of the advective MHD flux \cite{Torrilhon_2003}. A common solution is to regularize the equations by adding a vanishing viscous term, see e.g., \cite{Harten_1998,Mabuza_2020,Dao2022a}. For that purpose, a straightforward choice of the viscous regularization would be taking the resistive model of the MHD equations. However, the resistive MHD flux suffers from serious downsides. \cred{For example, unless the thermal diffusivity is zero, the resistive MHD flux is incompatible with the minimum entropy principle \cite{Guermond_Popov_2014}}.

One way to overcome these issues is to, instead, consider a monolithic regularization to the MHD equations \cite{Dao2022a}. The limitation with the monolithic flux includes that it is incompatible with the divergence-free nature of the magnetic field, $\DIV \bB = 0$; it does not conserve angular momentum; and it is too simplified since an equal amount of viscosity is added to all conserved variables. Therefore, we propose a more general and improved viscous regularization for the ideal MHD equations. Our idea is to combine the Guermond-Popov (GP) viscous flux \cite{Guermond_Popov_2014} for compressible flows with the resistive flux for the magnetic component of the MHD equations. In contrast to the monolithic flux, the GP flux can be physically motivated. On the other hand, the resistive flux for the electromagnetic part is suitable for many numerical schemes with mechanisms to guarantee $\DIV \bB = 0$. Our main focus is to carefully investigate positivity, conservation properties, entropy principles, and Galilean and rotational invariance of the new viscous regularization at the PDE level.




One of the conserved quantities we investigate more closely is angular momentum which is believed by physicists to be an important conserved property of fluid models. It is known that viscous regularizations of compressible flow can cause the model to lose conservation of angular momentum due to added diffusion of mass \cite{ottinger2009,Svard2018}. Additionally, since many numerical methods include so-called artificial mass diffusivity as stabilization, it is challenging to find numerical approximations that conserve angular momentum. In this work, we show that the GP flux \cite{Guermond_Popov_2014} does not conserve angular momentum and our proposed viscous flux conserves angular momentum while also satisfying entropy principles.





When $\DIV\bB\neq 0$, different techniques are needed when solving the MHD equations. To maintain thermodynamic consistency, a nonconservative source term proportional to the divergence is usually added, for example the Powell term \cite{Powell_et_al_1999}. Unless a specialized mechanism is built into the scheme, it is often necessary to further clean the divergence using, \eg the generalized Lagrange multiplier (GLM) \cite{Dedner_et_al_2002}. Whether such additions, \eg of the GLM or the Powell term, take away the good properties of the viscous regularization is also investigated in this work.




All of our findings are summarized in Tables~\ref{table:conclusion_viscous_model}, \ref{table:conclusion_divergence} and \ref{table:conclusion_GLM}. Overall, the theory and the numerical experiments suggest that the proposed viscous fluxes are suitable for artificial viscosity methods as well as \cred{being a viable candidate} for modeling fluid viscosity and magnetic resistivity.


The remaining sections are organized as follows. In Section~\ref{Sec:Equation}, we describe the ideal MHD equations and the nonconservative divergence source terms. The proposed viscous flux and the main analysis of positivity and entropy principles are presented in Section~\ref{sec:GP_flux}. Different conservativeness properties are discussed in Section~\ref{sec:conservativeness} where a slightly modified viscous flux is presented which preserves angular momentum. Several relevant numerical results are shown in Section~\ref{sec:numerical_results}. Section~\ref{sec:summary} is the conclusion.


\section{The ideal MHD equations}\label{Sec:Equation}
Consider the spatial domain $\mR^d, d=1,2,3$ and a temporal domain $[0,T]\subset\mR$. We define a vector of conserved quantities $\bsfU\coloneqq (\rho, \bbm^\top, E, \bB^\top)^\top$, where $\rho(\bx,t)\,:\, \mR^d\times[0,T]\rightarrow\mR$ is the \textit{density}, $\bbm(\bx,t)\,:\, \mR^d\times[0,T]\rightarrow\mR^d$ is the \textit{momentum}, $E(\bx,t)\,:\,\mR^d\times[0,T]\rightarrow\mR$ is the \textit{total energy}, and $\bB(\bx,t)\,:\, \mR^d\times[0,T]\rightarrow\mR^d$ is the \textit{magnetic field}. The system of MHD equations reads
\begin{equation}\label{eq:mhd1}
\p_t \bsfU + \DIV \bsfF_{\calE}(\bsfU) + \DIV \bsfF_{\calB}(\bsfU) = 
0,
\end{equation}
where the nonlinear advective fluxes $\bsfF_{\calE}, \bsfF_{\calB}$ are defined as
\begin{equation}\label{eq:mhd2}
  \bsfF_{\calE}\coloneqq
  \begin{pmatrix}
  \bbm \\
  \bbm \otimes \bu + p\polI \\
  \bu (E+p) \\
  0
  \end{pmatrix},
  \quad
  \bsfF_{\calB}\coloneqq
  \begin{pmatrix}
  0 \\
  -\bbetaa\\
  -\bbetaa\bu\\
  \bu \otimes \bB - \bB \otimes \bu \\
    \end{pmatrix}.
\end{equation}
The symmetric term $\bbetaa\coloneqq -\frac{1}{2} (\bB\SCAL\bB)\polI+\bB\otimes\bB $ is the \textit{Maxwell stress tensor}. Aligned with \cite{Guermond_Popov_2014}, the \textit{pressure} $p$ is not assumed to be positive. We respectively refer to the four conservation equations in \eqref{eq:mhd1} as: the \textit{mass equation}, the \textit{momentum equation}, the \textit{total energy equation}, and the \textit{magnetic equation}. The specific internal energy $e$ is defined as
\begin{equation}\label{eq:e}
  e \coloneqq \rho^{-1}E-\frac12|\bu|^2-\frac12\rho^{-1}|\bB|^2.
\end{equation}
We consider a general equation of state,
\begin{equation}\label{eq:equation_of_state}
p s_e + \rho^2 s_{\rho} = 0,
\end{equation}
where $s$ is the \textit{specific entropy}, $s_e\coloneqq\p_es, s_{\rho}\coloneqq\p_{\rho}s$. The \textit{temperature} $T$ is defined as $T\coloneqq s_e^{-1}$. In this work, we assume that $-s$ is a strictly convex function with respect to $\rho^{-1}$ and $e$, and $T$ is positive. The ideal MHD equations can be regularized by adding a parabolic term $\DIV \bsfF_{\calV}(\bsfU)$ to the right hand side of \eqref{eq:mhd1},
\begin{equation}\label{eq:mhd_viscous}
  \p_t \bsfU + \DIV \bsfF_{\calE}(\bsfU) + \DIV \bsfF_{\calB}(\bsfU) = \DIV \bsfF_{\calV}(\bsfU).
\end{equation}
For example, one can simply choose $\bsfF_{\calV}(\bsfU)$ to be the \textit{monolithic parabolic flux} $\bsfF_{\calV}^{\text{m}}(\bsfU)$ $\coloneqq$ $\epsilon\GRAD\bsfU$ where $\epsilon$ is a positive vanishing coefficient. A continuous analysis of $\bsfF_{\calV}^{\text{m}}(\bsfU)$ has been done in \cite{Dao2022a}.

\subsection{The nonconservative divergence source term}\label{sec:div_source_term}
An important physical property of the magnetic field is that the divergence of it is pointwise zero. However, without any special consideration, numerical solutions to \eqref{eq:mhd1} do not satisfy this property exactly. In that case, several physical properties related to the magnetic field are violated, and numerical methods may converge to wrong solutions, see e.g., \cite{Brackbill_Barnes_1980}. When $\DIV\bB\neq 0$ is unavoidable, ensuring thermodynamic consistency often involves adding a nonconservative term in the form of,
\begin{equation}\label{eq:source_term}
\Psi(\alpha_{\bbm},\alpha_E,\alpha_{\bB}) \coloneqq \begin{pmatrix}
  0\\
  \alpha_{\bbm}\bB\\
  \alpha_{E}\bu\SCAL\bB\\
  \alpha_{\bB}\bu
\end{pmatrix}(\DIV\bB),
\end{equation}
where $\alpha_{\bbm},\alpha_E,\alpha_{\bB}$ are constants. Popular divergence source terms are $\Psi_{\text{Powell}}$ $\coloneqq$ $\Psi(-1,-1,-1)$ \cite{Powell_et_al_1999}, $\Psi_{\text{Janhunen}}$ $\coloneqq$ $\Psi(0,0,-1)$ \cite{Janhunen_2000}, and $\Psi_{\text{BB}}$ $\coloneqq$ $\Psi(-1,0,0)$ \cite{Brackbill_Barnes_1980}. With the Powell source term being the most popular choice, symmetrization \cite{Godunov_1972}, consistency with the Lorentz force, and entropy principles of the MHD equations are recovered when $\DIV\bB\neq0$.
For the rest of the paper, we do not assume that $\DIV\bB = 0$.

\section{A Guermond-Popov (GP) type viscous flux}\label{sec:GP_flux}
Our idea is to regularize the hydrodynamic part of \eqref{eq:mhd1} following the work of \cite{Guermond_Popov_2014} and electromagnetic part using the resistive MHD flux, see e.g., \cite{Rueda-Ramirez_2021}. We propose the viscous flux $\bsfF_{\calV}(\bsfU)$ in \eqref{eq:mhd_viscous} to be $\bsfF_{\calV}^{\text{GP}}(\bsfU)$ defined as
\begin{equation}\label{eq:GP_mhd_flux}
\bsfF_{\calV}^{\text{GP}}(\bsfU)\coloneqq
\begin{pmatrix}
\kappa\nabla\rho \\
\mu\rho\GRAD^s\bu+(\kappa\nabla\rho)\otimes\bu \\
\kappa\GRAD(\rho e) + \frac{|\bu|^2}{2}\kappa\nabla\rho + \mu\rho (\GRAD^s\bu) \bu + \eta \big( \GRAD \bB - \GRAD \bB^\top \big) \bB\\
\eta \big( \GRAD \bB - \GRAD \bB^\top \big)
\end{pmatrix},
\end{equation}
where $\kappa > 0, \mu > 0, \eta > 0$ are viscosity coefficients, and $\GRAD^s \bu \coloneqq \frac{1}{2} \l( \GRAD \bu + (\GRAD \bu)^\top \r) $. The resulting GP-MHD system reads
\begin{equation}\label{eq:mhd_GP_Powell}
\p_t \bsfU + \DIV \bsfF_{\calE}(\bsfU) + \DIV \bsfF_{\calB}(\bsfU) = \DIV \bsfF_{\calV}^{\text{GP}}(\bsfU) + \Psi.
\end{equation}
This section analyzes several physically relevant properties of the viscosity regularized system \eqref{eq:mhd_GP_Powell}.

\subsection{Positivity of density}

\begin{theorem}[Positivity of density]\label{thm:positivity_density}
  Assume sufficient smoothness of $\kappa$, $\rho$, $\bu$, and boundedness of $\bu$, $\DIV\bu$, $\kappa\GRAD\rho$, $\p_t\rho+\DIV(\rho\bu)$. With the addition of the viscous term $\DIV(\kappa\nabla\rho)$ to the mass equation, positivity of density is guaranteed at the continuous level, that is
  \[
  \rho(\bx, t) > 0, \text{ for } \bx \in \mR^d, t > 0.
  \]
\end{theorem}

Since the resulting mass equation coincides with the Euler mass equation regularized by the GP flux, see \cite{Guermond_Popov_2014} for the proof of Theorem~\ref{thm:positivity_density}. Positivity of density plays a fundamental role in ensuring other thermodynamic properties of \eqref{eq:mhd_GP_Powell}. By specific choices of the vanishing coefficient $\kappa$, it can be shown that the regularized mass equation is a continuous analog of the Lax-Friedrichs scheme or the upwind scheme solving the unregularized mass equation, see e.g., \cite[Section 3]{Dao2022a}. In Section~\ref{sec:contact_line}, we numerically demonstrate the importance of having artificial mass diffusion to obtain physically correct solutions.

\subsection{Minimum entropy principle}
 We introduce the following notations,
\[
\begin{aligned}
  \bbf & = \kappa\GRAD\rho, &&& \bl & = \kappa\GRAD(\rho e),\\
  \polG &= \mu\rho\GRAD^s\bu, &&& \bh & = \bl-\frac{|\bu|^2}{2}\bbf,\\
  \polg & = \polG+\bbf\otimes\bu, &&& \polk & = \eta \big( \GRAD \bB - \GRAD \bB^\top \big).
\end{aligned}
\]
The viscous flux \eqref{eq:GP_mhd_flux} can be conveniently written as
\begin{equation}\label{eq:GP_mhd_flux_2}
\bsfF_{\calV}^{\text{GP}}(\bsfU)\coloneqq
\begin{pmatrix}
\bbf \\
\polg \\
\bh + \polg\SCAL\bu+\polk\SCAL\bB\\
\polk
\end{pmatrix}.
\end{equation}
An equation describing the evolution of the specific entropy is derived next.
\begin{lemma}\label{lemma:min_entropy_specific}
The specific entropy $s$ for the GP-MHD system \eqref{eq:mhd_GP_Powell} satisfies
\[
\rho (\p_t s + \bu \SCAL \nabla s) - \DIV (\rho \kappa \nabla s) - \bbf\SCAL\nabla(es_e-\rho s_{\rho}) + \bl\SCAL\nabla s_e-s_e(\polG:\GRAD\bu+\polk:\GRAD\bB) = 0
\]
given that $\alpha_E-\alpha_{\bbm}-\alpha_{\bB}=1$.
\end{lemma}
\begin{proof}
The following identity can be derived from the definitions of $\bl, \bbf$ and the chain rule $\GRAD s = s_\rho\GRAD\rho+s_e\GRAD e$, see \cite{Dao2022a},
\begin{equation}\label{eq:min_entropy_specific_l}
\bl=s_{e}^{-1}\left(e s_{e}-\rho s_{\rho}\right)\bbf+\kappa\rho s_{e}^{-1} \nabla s.
\end{equation}
We rewrite the momentum equation as
\[
\rho(\p_t\bu+\bu\SCAL\GRAD\bu) + \bu \DIV\bbf + \GRAD p - \DIV\bbetaa - \DIV(\mu\rho\GRAD^s\bu+(\kappa\GRAD\rho)\otimes\bu) = \alpha_\bbm \bB(\DIV\bB).
\]
Multiplying the momentum equation with $\bu$, the magnetic equation with $\bB$, then subtracting them from the energy equation gives
\begin{equation}\label{eq:min_entropy_e}
\rho(\p_te+\bu\SCAL\nabla e) +\left(e-\frac{1}{2}|\bu|^2\right)\DIV\bbf+p(\DIV\bu)+(\bB\SCAL\bu)(\DIV\bB) + \Upsilon_{e,V} = \Upsilon_{e,P},
\end{equation}
where the contribution of the viscous flux is
\begin{align*}
\Upsilon_{e,V} & =
-\DIV(\bh+\polg\SCAL\bu+\polk\SCAL\bB)+(\DIV\polg)\SCAL\bu+(\DIV\polk)\SCAL\bB\\
& =-\DIV\bh - \polg:\GRAD\bu - \polk:\GRAD\bB,
\end{align*}
and the contribution of the Powell term is
\[
\Upsilon_{e,P}=
(\alpha_E-\alpha_{\bbm}-\alpha_{\bB})(\bu\SCAL\bB)(\DIV\bB).
\]
If $\alpha_E-\alpha_{\bbm}-\alpha_{\bB}=1$, $\Upsilon_{e,P}$ is cancelled out by $(\bB\SCAL\bu)(\DIV\bB)$ from the left hand side of \eqref{eq:min_entropy_e}. Utilizing the chain rule $\p_\alpha s = s_\rho\p_\alpha\rho+s_e\p_\alpha e, \alpha\in\{t,\bx\}$, we multiply the density equation with $\rho s_\rho$, \eqref{eq:min_entropy_e} with $s_e$ to derive the following entropy conservation equation,
\begin{align*}
 & \rho\p_t s + \rho\bu\SCAL\GRAD s + \rho^2 s_\rho(\DIV\bu)+s_e(e-\frac12|\bu|^2)\DIV\bbf \\
 & +ps_e(\DIV\bu)-s_e(\DIV\bh+\polg:\GRAD\bu+\polk:\GRAD\bB)-\rho s_\rho\DIV\bbf = 0.
\end{align*}
Due to the equation of state \eqref{eq:equation_of_state}, the terms associating with $\DIV\bu$ are cancelled,
\[
\rho(\p_ts+\bu\SCAL\GRAD s) + (es_e-\rho s_\rho)\DIV\bbf-s_e\DIV\bh-s_e\polg:\GRAD\bu-s_e\polk:\GRAD\bB-s_e\frac{1}{2}|\bu|^2\DIV\bbf=0.
\]
We rewrite the equation by applying a product rule on $-s_e\frac{1}{2}|\bu|^2\DIV\bbf$,
\begin{align*}
\rho(\p_ts+\bu\SCAL\GRAD s) + (es_e-\rho s_\rho)\DIV\bbf-s_e\DIV\left(\bh+\frac12|\bu|^2\bbf\right) &\\
-s_e(\polg:\GRAD\bu-(\bbf\otimes\bu):\GRAD\bu+\polk:\GRAD\bB) &= 0.
\end{align*}
Because $\polg-(\bbf\otimes\bu) = \polG$, we end up with
\[
\rho(\p_ts+\bu\SCAL\GRAD s) + (es_e-\rho s_\rho)\DIV\bbf-s_e\DIV\left(\bh+\frac12|\bu|^2\bbf\right)
-s_e(\polG:\GRAD\bu+\polk:\GRAD\bB) = 0.
\]
Applying a product rule on $(es_e-\rho s_\rho)\DIV\bbf$ and using \eqref{eq:min_entropy_specific_l}, we obtain the equality of the statement of the lemma.
\end{proof}
\begin{lemma}\label{lemma:min_entropy_quadratic_form}
The quadratic form
\[
J_1(\GRAD\rho,\GRAD e) \coloneqq -\bbf\SCAL\nabla(es_e-\rho s_{\rho}) + \bl\SCAL\GRAD s_e + \kappa\GRAD\rho\SCAL\GRAD s
\]
is negative definite.
\end{lemma}
\begin{proof}
Using the chain rule $\GRAD s = s_\rho\p_\alpha\rho+s_e\p_\alpha e$, we can rewrite $J_1$ as
\[
J_1 = 
\begin{pmatrix}
  \nabla\rho\\
  \nabla e
\end{pmatrix}
\left(
\begin{pmatrix}
\epsilon\rho^{-1}\p_{\rho}(\rho^2 s_{\rho}) & \epsilon\rho s_{\rho e} \\
\epsilon\rho s_{\rho e} & \epsilon\rho s_{ee}  
\end{pmatrix}
\otimes
\polI_d
\right)
\begin{pmatrix}
\nabla\rho & \nabla e
\end{pmatrix}.
\]
Due to the strict convexity of $-s$, it can be shown that the $2$-by-$2$ matrix
\[
\begin{pmatrix}
\epsilon\rho^{-1}\p_{\rho}(\rho^2 s_{\rho}) & \epsilon\rho s_{\rho e} \\
\epsilon\rho s_{\rho e} & \epsilon\rho s_{ee}  
\end{pmatrix}
\]
is negative definite because its determinant is positive and its trace is negative, see \citep[Appendix A]{Dao2022a}.
\end{proof}
\begin{theorem}[Minimum entropy principle]\label{thm:min_entropy}
Assume sufficient smoothness and that the density and the internal energy uniformly converge to stationary constant states $\rho^*, e^*$ outside of a compact set $\Omega\in\mR^d$. The GP-MHD system \eqref{eq:mhd_GP_Powell} exhibits a specific entropy function $s$ which satisfies
\begin{equation}\label{eq:min_entropy}
\inf_{\bx\in\mR^d}s(\bx,t) \geq \inf_{\bx\in\mR^d}s_0(\bx).
\end{equation}
\end{theorem}

\begin{proof}
We first prove that $\polk:\nabla \bB \geq 0$. Using that the contraction between a symmetric and anti-symmetric matrix with zero diagonal is zero \citep[Ch 11.2.1]{Larson_2013}, one can show that
\begin{equation} \label{eq:resisitivity_sign}
  \eta\left(\nabla \bB - \nabla \bB^\top\right) : \nabla \bB  = \frac12\eta\left(\nabla \bB - \nabla \bB^\top\right) : \left(\nabla \bB - \nabla \bB^\top\right) \geq 0. 
\end{equation}
Isolating the signed terms in the equation of Lemma~\ref{lemma:min_entropy_specific} and using Lemma~\ref{lemma:min_entropy_quadratic_form} and \eqref{eq:resisitivity_sign}, we have
\begin{equation}\label{eq:entropy_equality}
\rho (\p_t s + \bu \SCAL \nabla s) - \DIV (\rho \kappa \nabla s) -\kappa\GRAD\rho\SCAL\GRAD s = -J_1 + s_e(\polG:\GRAD\bu+\polk:\GRAD\bB) \geq 0.
\end{equation}
At time $t > 0$, consider the case $\inf_{\bx\in\mR^d}s(\bx,t)$ is reached outside of $\Omega$. The result \eqref{eq:min_entropy} follows straightforwardly because
\[
\inf_{\bx\in\mR^d}s(\bx,t) = s^* \geq \inf_{\bx\in\mR^d}s_0(\bx).
\]
Otherwise, if $\inf_{\bx\in\mR^d}s(\bx,t)$ is reached at a point $\bar{\bx}(t)\in\Omega$, we have $\GRAD s(\bar{\bx}(t), t) = 0$ and $\LAP s(\bar{\bx}(t), t) \geq 0$ due to the smoothness assumption. At $\bar{\bx}(t)\in\Omega$, from \eqref{eq:entropy_equality}, we have
\[
\rho(\bar{\bx}(t), t)\p_t s(\bar{\bx}(t), t) - \kappa\rho(\bar{\bx}(t), t)\LAP s(\bar{\bx}(t), t) \geq 0.
\]
Therefore, $\p_t \inf_{\bx\in\mR^d}s(\bx, t) = \p_t s(\bar{\bx}(t), t) \geq 0$ because $\rho > 0$. This concludes that $\inf_{\bx\in\mR^d}s(\bx,t)$ does not decrease as the solution evolves in time. The proof is complete.
\end{proof}

\subsection{Positivity of internal energy}

\begin{theorem}[Positivity of internal energy]\label{thm:positivity_internal_energy}
  The specific internal energy $e$ and the internal energy $\rho e$ of the GP-MHD system \eqref{eq:mhd_GP_Powell} always remain positive given that the initial data fulfills $e > 0$.
\end{theorem}

\begin{proof}
From the assumption that the temperature is positive, $T = s_e^{-1} > 0$, we have $s_e > 0$ where $s = s(e, \rho^{-1})$ and $e = e(s,\rho^{-1})$ defined by the equation of state \eqref{eq:equation_of_state}. Given positivity of density $\rho > 0$, we have $s_e > 0$ leads to $e_s > 0$ \citep[Appendix A.1]{Guermond_Popov_2014}. The minimum entropy principle implies that $s(\bx,t) \geq \inf_{\mR^d}s(\bx,t)\geq \inf_{\mR^d}s_0(\bx)$, where $s_0$ is the specific internal energy at $t=0$, which leads to
\[
e(s,\rho^{-1}) \geq e(\inf_{\mR^d}s,\rho^{-1}) \geq e(\inf_{\mR^d}s_0,\rho^{-1})
\]
for all $\rho > 0$ since $e_s > 0$. Therefore, given physically correct initial data, the specific internal energy is positive at all times $t > 0$, and so is the internal energy $\rho e$ due to the proven positivity of density.
\end{proof}

If an ideal gas is considered, Theorems~\ref{thm:positivity_density} and \ref{thm:positivity_internal_energy} give positivity of pressure following the ideal equation of state $p = (\gamma-1)\rho e$, where $\gamma$ is the \textit{adiabatic gas constant}.

\subsection{Generalized entropy inequalities}

In this section, we show that the proposed regularized model is compatible with all the generalized entropy inequalities. The generalized entropies in the manner of \cite{Harten_1998} consider a large class of strictly convex entropies $-\rho\phi(s)$ where $\phi(s)$ is a twice differentiable function. We say that $S=-\rho\phi(s)$ is strictly convex if and only if the Hessian matrix $S_{\bsfU\bsfU}$ is positive definite. In the following theorem, we look at important properties of $\phi(s)$.

\begin{theorem}[Properties of strictly convex generalized entropies \cite{Dao2022a}]\label{thm:generalized_entropy_convexity}
The generalized entropy $-\rho\phi(s)$ is strictly convex if and only if
\[
\phi'(s) > 0 \quad\text{ and }\quad
\frac{\phi'(s)}{c_p}-\phi''(s) > 0,
\]
where $c_p\coloneqq T\frac{\p s(p,T)}{\p T}$ is the specific heat capacity at constant pressure.
\end{theorem}

The result of Theorem~\ref{thm:generalized_entropy_convexity} is derived by investigating the signs of the eigenvalues of $S_{\bsfU\bsfU}$ where the necessary and sufficient condition for the eigenvalues to be positive is the two inequalities of Theorem~\ref{thm:generalized_entropy_convexity}, see \citep[Appendix B]{Dao2022a}.

The main result of this section is stated in the following theorem.

\begin{theorem}[Generalized entropy inequalities]\label{thm:generalized_entropy} All smooth solutions to \eqref{eq:mhd_GP_Powell} satisfy
\[
\p_t(\rho \phi(s))+\DIV(\bu\rho \phi(s)-\kappa\rho\GRAD \phi(s)-\kappa \phi(s)\GRAD\rho) \geq 0.
\]
\end{theorem}

\begin{proof}
Multiply \eqref{eq:entropy_equality} with $\phi'(s)$, we have
\begin{equation}\label{eq:generalized_entropy_phi_prime}
\begin{aligned}
  \rho (\p_t \phi(s)+\bu\SCAL\GRAD\phi(s)) - \DIV(\kappa\rho\GRAD\phi(s)) + \kappa\rho\phi''(s)|\GRAD s|^2-\kappa\phi'(s)\GRAD\rho\SCAL\GRAD s\\
  +\phi'(s)J_1 = \phi'(s)s_e\polG:\GRAD\bu+\phi'(s)s_e\polk:\GRAD\bB.
\end{aligned}
\end{equation}
Multiplying the regularized mass equation with $\phi$ then adding the result to \eqref{eq:generalized_entropy_phi_prime}, the product rules give
\begin{equation}\label{eq:generalized_entropy_equality}
\begin{aligned}
 & \p_t(\rho\phi(s)) + \DIV(\bu\rho\phi(s))-\DIV(\kappa\rho\GRAD\phi(s)+\kappa\phi(s)\GRAD\rho)\\
 &  =-\kappa\rho\phi''(s)|\GRAD s|^2-\phi'(s)J_1 + \phi'(s)s_e\polG:\GRAD\bu+\phi'(s)s_e\polk:\GRAD\bB.
\end{aligned}
\end{equation}
Using Theorem~\ref{thm:generalized_entropy_convexity}, we have
\begin{equation}\label{eq:generalized_entropy_J1J2}
-\kappa\rho\phi''(s)|\GRAD s|^2-\phi'(s)J_1 > -\phi'(s)(\kappa\rho c_p^{-1}|\GRAD s|^2+J_1).
\end{equation}
Let $J_2 \coloneqq \kappa\rho c_p^{-1}|\GRAD s|^2+J_1$. It is possible to write $J_2$ in a quadratic form as
\[
J_2 = \kappa
\begin{pmatrix}\GRAD\rho\\ \GRAD e\end{pmatrix}
\begin{pmatrix}
  J_3
  \otimes
  \polI_d
\end{pmatrix}
\begin{pmatrix}\GRAD\rho & \GRAD e\end{pmatrix},
\]
where $J_3$ is obtained by applying the chain rule on $\GRAD s$,
\[
J_3 =
\begin{pmatrix}
c_p^{-1}\rho s_\rho^2+\rho^{-1}\p_{\rho}(\rho^2 s_{\rho}) & c_p^{-1}s_\rho\rho s_e+\rho s_{\rho e} \\
c_p^{-1}\rho s_\rho s_e+\rho s_{\rho e} & c_p^{-1}\rho s_e^2+ \rho s_{ee}
\end{pmatrix}.
\]
The $2\times2$ matrix $J_3$ is negative definite due to the strict convexity of $-s$, see \citep[Lemma 4]{Dao2022a}. We deduce from this and \eqref{eq:generalized_entropy_J1J2} that
\[
-\kappa\rho\phi''(s)|\GRAD s|^2-\phi'(s)J_1 \geq 0.
\]
Therefore, the right hand side of \eqref{eq:generalized_entropy_equality} is greater than or equal to zero, which leads to the desired conclusion.
\end{proof}

\subsection{Galilean and rotational invariance}
By Galilean relativity, a physical law should stay invariant under a Galilean transformation, which in this case means that the set of equations does not change under a shift of the reference frame. In the following theorem, we state this property of the GP-MHD system.
\begin{theorem}\label{thm:galilean_invariance}
The GP-MHD system \eqref{eq:mhd_GP_Powell} is Galilean invariant.
\end{theorem}

\begin{proof}
See Section~\ref{appendix:galilean_invariance}.
\end{proof}

Rotational invariance is also a physically relevant property. This property says that the laws described by the set of equations remain valid when the observation frame is rotated by an arbitrary angle. Regarding numerical methods, many Riemann solver-based methods rely on rotational invariance, see e.g., \cite{Billett_Toro_1998}. We look at this property of the GP-MHD system in the following theorem.

\begin{theorem}\label{thm:rotational_invariance}
The GP-MHD system \eqref{eq:mhd_GP_Powell} is rotationally invariant.
\end{theorem}

\begin{proof}
See Section~\ref{appendix:rotational_invariance}.
\end{proof}

\subsection{GLM-GP-MHD systems}
Including the Powell term alone is insufficient to reduce $|\DIV\bB|$ in general. As the solution evolves, it is necessary to clean up the unphysical portion of $\bB$. A preferred approach to divergence cleaning is the hyperbolic cleaning method, see e.g., \cite{Bohm_2020,Guillet_2019,Rueda-Ramirez_2021}, also known as the generalized Lagrange multiplier (GLM) \cite{Dedner_et_al_2002}. In this section, we investigate three variations of the GLM with the most favorable conservative properties. The first one is a classical formulation by \cite{Dedner_et_al_2002}. The second one is an improvement by \cite{Derigs_2018}. The third one is a modified version of the second one which compromises Galilean invariance in exchange for energy conservation.

\subsubsection{The Galilean invariant extended GLM system by \cite{Dedner_et_al_2002}}
 The GLM system is obtained by adding a conserved variable $\Phi$ with an extra governing equation,
\[
\p_t \Phi+c_{h} \DIV \bB=-\bu\SCAL\GRAD\Phi-\frac{c_r c_h}{h} \Phi,
\]
where $c_h$ is the divergence cleaning speed, $c_r$ is a global constant, and $h$ is a mesh size indicator. We note that the GLM system being used here is referred to as the Galilean invariant extended GLM system in \cite{Dedner_et_al_2002,Derigs_2018}. The vector of conserved variables is extended to $\bsfU \coloneqq (\rho, \bbm^\top, E, \bB^\top, \Phi)^\top$. The GLM-GP-MHD system reads
\begin{equation}\label{eq:mhd_GP_Powell_GLM}
\p_t \bsfU + \DIV \bsfF_{\calE}(\bsfU) + \DIV \bsfF_{\calB}(\bsfU) - \DIV \bsfF_{\calV}^{\text{GP}}(\bsfU) = \Psi + \Upsilon_{\text{GLM}},
\end{equation}
where all the GLM terms are gathered in $\Upsilon_{\text{GLM}}$,
\[
\Upsilon_{\text{GLM}} \coloneqq \begin{pmatrix}
0\\
0\\
-c_h\bB\SCAL\GRAD\Phi\\
-c_h\GRAD\Phi\\
-\bu\SCAL\GRAD\Phi-\frac{c_rc_h}{h} \Phi-c_h\DIV \bB
\end{pmatrix}.
\]
\begin{theorem}\label{thm:dedner_properties}
The results of Theorems~\ref{thm:positivity_density} (positivity of density), \ref{thm:min_entropy} (minimum entropy principle), \ref{thm:positivity_internal_energy} (positivity of internal energy), and \ref{thm:generalized_entropy} (generalized entropy inequalities) still hold for the GLM-GP-MHD system \eqref{eq:mhd_GP_Powell_GLM}.
\end{theorem}

\begin{proof}
Theorem~\ref{thm:positivity_density} follows because the mass equation remains the same in \eqref{eq:mhd_GP_Powell_GLM}.

We now revisit a key step in Theorem~\ref{thm:min_entropy} with \eqref{eq:mhd_GP_Powell_GLM}. Multiplying the $\Upsilon_{\text{GLM}}$ components in the momentum equation with $\bu$, the magnetic equation with $\bB$, then subtract them from the $\Upsilon_{\text{GLM}}$ component in the energy equation gives zero. Therefore, \eqref{eq:min_entropy_e} also holds for the GLM-GP-MHD system \eqref{eq:mhd_GP_Powell_GLM} and so the rest of the proof to Theorem~\ref{thm:min_entropy}, and so Theorem~\ref{thm:generalized_entropy}. Theorem~\ref{thm:positivity_internal_energy} holds since both Theorems~\ref{thm:positivity_density} and \ref{thm:min_entropy} hold.
\end{proof}

Furthermore, the Galilean invariance of $\Upsilon_{\text{GLM}}$ is known, see \cite{Dedner_et_al_2002}. Thus, Galilean invariance of the GLM-GP-MHD system \eqref{eq:mhd_GP_Powell_GLM} naturally follows.

\subsubsection{Nine-wave formulation by \cite{Derigs_2018}}
By adding a contribution of $\Phi$ into the total energy
\[
E^* = \rho e + \frac12\rho|\bu|^2+\frac12|\bB|^2+\frac12\Phi^2,
\]
a new GLM formulation is derived by \cite{Derigs_2018},
\[
\Upsilon_{\text{9W-GLM}} \coloneqq \begin{pmatrix}
0\\
0\\
-c_h\bB\SCAL\GRAD\Phi - c_h\Phi(\DIV\bB) + \Phi\bu\SCAL\GRAD\Phi\\
-c_h\GRAD\Phi\\
-\bu\SCAL\GRAD\Phi-c_h\DIV \bB
\end{pmatrix}
=
\begin{pmatrix}
0\\
0\\
-c_h\DIV(\Phi\bB) + \Phi\bu\SCAL\GRAD\Phi\\
-c_h\GRAD\Phi\\
-\bu\SCAL\GRAD\Phi-c_h\DIV \bB
\end{pmatrix}.
\]
The properties in Theorem~\ref{thm:dedner_properties} and Galilean invariance hold true for $\Upsilon_{\text{9W-GLM}}$ with $E$ being replaced by $E^*$. We skip the proofs here because except for the positivity of density, $\Upsilon_{\text{9W-GLM}}$ has been thoroughly analyzed in \cite{Derigs_2018}.

\subsubsection{An energy conservative GLM formulation} \label{Sec:energy_GLM}
If conservation of energy is of interest, we propose a new GLM formulation which is a slight modification of the GLM formulation proposed by \cite{Derigs_2018},
\[
\Upsilon_{\text{CONS-GLM}} \coloneqq 
\begin{pmatrix}
0\\
0\\
-c_h\DIV(\Phi\bB) \\
-c_h\GRAD\Phi\\
-c_h\DIV \bB
\end{pmatrix}.
\]
Similar to $\Upsilon_{\text{9W-GLM}}$, the properties in Theorem~\ref{thm:dedner_properties} hold true for $\Upsilon_{\text{CONS-GLM}}$ with $E$ being replaced by $E^*$. However, $\Upsilon_{\text{CONS-GLM}}$ conserves $E^*$ but is not Galilean invariant and $\Upsilon_{\text{9W-GLM}}$ is Galilean invariant but does not conserve $E^*$.



\subsection{Conservation}\label{sec:conservativeness}
In this section we investigate how the viscous regularization $\bsfF_{\calV}(\bsfU)$ and divergence source term $\Psi(\alpha_{\bbm},\alpha_E,\alpha_{\bB})$ effect conservation properties. \cred{To simplify the boundary integrals, we assume that the velocity $\bu$ and the magnetic field $\bB$ are compactly supported. This assumption has no conflicts with the previous assumptions in Sections~\ref{thm:positivity_density} and \ref{thm:min_entropy}}. The ideal MHD equations \eqref{eq:mhd1} with pointwise $\DIV \bB = 0$ conserve mass, momentum, angular momentum, total energy, magnetism and magnetic helicity. These are defined as
\begin{equation} \label{eq:conservation_properties}
\begin{alignedat}{4}
& \text{Mass} \quad && \int_{\mR^d} \rho \ud \bx, \quad && \text{Momentum} \quad && \int_{\mR^d}  \bbm \ud \bx, \\
& \text{Total energy} &&\int_{\mR^d} E \ud \bx, && \text{Angular momentum} \quad && \int_{\mR^d}  \bbm \CROSS \bx \ud \bx, \\ 
& \text{Magnetism} \quad && \int_{\mR^d}  \bB  \ud \bx, && \text{Magnetic helicity} \quad && \int_{\mR^d}  \bA \SCAL \bB \ud \bx, 
\end{alignedat}
\end{equation}
where $\bA$ is a vector potential to $\bB$, \ie $\nabla \CROSS \bA = \bB$. For $\bB$ to admit a vector potential $\bA$, we require that $\bB$ is pointwise divergence-free since the divergence of curl is zero, \ie $ \DIV ( \nabla \CROSS \bA ) = 0$. Thus magnetic helicity is only conserved if $\DIV \bB = 0$. The other properties in \eqref{eq:conservation_properties} can be shown to be conserved for the ideal MHD equations \eqref{eq:mhd1} even if $\DIV \bB \neq 0$.



The viscous regularization $\bsfF_{\calV}(\bsfU)$ inside \eqref{eq:mhd_viscous} can effect the conserved properties. It is known that zero resistivity in the magnetic equation is required for the MHD equations to conserve magnetic helicity, see Section~\ref{appendix:magnetic_helicity} for more details. It is also known that viscous regularization via mass diffusion, \ie $\kappa \neq 0$, can make flow models lose conservation of angular momentum which is a criticism brought forward against such models \cite{ottinger2009,Svard2018}. In Section \ref{sec:angular_flux}, we propose a variant of the viscous regularization \eqref{eq:GP_mhd_flux} that conserves angular momentum. Lastly, the divergence source term $\Psi(\alpha_{\bbm},\alpha_E,\alpha_{\bB})$ will affect conservation depending on $\alpha_{\bbm},\alpha_E,\alpha_{\bB}$, see \eg \cite{Powell_et_al_1999,Janhunen_2000}.






\subsection{An angular momentum conserving viscous regularization} \label{sec:angular_flux}
By modifying the GP flux such that the viscous flux in the momentum equations is symmetric one can ensure that angular momentum is conserved. It is, however, also important that entropy principles still hold, meaning that the energy equation has to be adequately compensated. To this end, if conservation of angular momentum is of interest, we propose the following $\text{GP}^s$-MHD system
\begin{equation}\label{eq:mhd_angular_Powell}
\p_t \bsfU + \DIV \bsfF_{\calE}(\bsfU) + \DIV \bsfF_{\calB}(\bsfU) = \DIV \bsfF^{\text{GP}^s}_{\calV}(\bsfU) + \bsfF^E(\bsfU)  + \Psi,
\end{equation}
where $\bsfF^{\text{GP}^s}_{\calV}$ is the same as $\bsfF^{\text{GP}}_{\calV}$ except that the viscous flux in the momentum equations is symmetric, $\bsfF^E$ is a nonconservative contribution to the energy equation, and $\Psi$ is the nonconservative divergence source term defined in Section~\ref{sec:div_source_term}. The new terms $\bsfF^{\text{GP}^s}_{\calV}$ and $\bsfF^E$ are given by
\begin{equation}
\bsfF^{\text{GP}^s}_{\calV}(\bsfU) \coloneqq \begin{pmatrix}
\kappa\nabla\rho \\
\mu\rho\GRAD^s\bu       +    \frac{1}{2} ( (\kappa\nabla\rho)\otimes\bu + \bu \otimes (\kappa\nabla\rho)  )    \\
\kappa\GRAD(\rho e) + \frac{|\bu|^2}{2}\kappa\nabla\rho + \mu\rho (\GRAD^s\bu)\bu + \eta \big( \GRAD \bB - \GRAD \bB^\top \big) \bB \\
\eta \big( \GRAD \bB - \GRAD \bB^\top \big)
\end{pmatrix},
\end{equation}
\begin{equation} \label{eq:energy_compensation}
\bsfF^E \coloneqq \begin{pmatrix}
0 \\
0    \\
  \frac{1}{2}  \l(   \DIV \l( \bu \otimes (\kappa\nabla\rho) - (\kappa\nabla\rho) \otimes \bu \r) \r) \SCAL \bu \\
0
\end{pmatrix}.
\end{equation}


We note that in one spatial dimension the $\text{GP}^s$-MHD system \eqref{eq:mhd_angular_Powell} is equivalent to the GP-MHD system \eqref{eq:mhd_GP_Powell}. In more than one spatial dimension they are different in the following ways. The $\text{GP}^s$-MHD system conserves angular momentum, but not total energy whereas the GP-MHD system conserves total energy but not angular momentum. Both systems satisfy the same entropy principles. This is summarized in Theorems \ref{thm:entropy_ang} and \ref{thm:conservation}.




\begin{theorem} \label{thm:entropy_ang}
The results of Theorems~\ref{thm:positivity_density} (positivity of density), \ref{thm:min_entropy} (minimum entropy principle), \ref{thm:positivity_internal_energy} (positivity of internal energy), and \ref{thm:generalized_entropy} (generalized entropy inequalities) still hold for the $\text{GP}^s$-MHD system \eqref{eq:mhd_angular_Powell}.
\end{theorem}

\begin{proof}
Theorem~\ref{thm:positivity_density} follows because the mass equation remains the same in \eqref{eq:mhd_angular_Powell}.

We now revisit a key step in Lemma \ref{lemma:min_entropy_specific} with \eqref{eq:mhd_angular_Powell}. By rewriting the momentum equations as
\begin{align*}
&  \rho(\p_t\bu+\bu\SCAL\GRAD\bu) + \bu \DIV\bbf + \GRAD p - \DIV\bbetaa \\
&  - \DIV(\mu\rho\GRAD^s\bu+(\kappa\GRAD\rho)\otimes\bu)  - \DIV \l(  \frac{1}{2} \l(   \bu \otimes \kappa\GRAD\rho  - \kappa\GRAD\rho \otimes\bu \r) \r)= \alpha_\bbm \bB(\DIV\bB).
\end{align*}
Multiplying them with $\bu$, the magnetic equations with $\bB$, then subtracting them from the energy equation gives \eqref{eq:min_entropy_e}. Therefore, \eqref{eq:min_entropy_e} also holds for the $\text{GP}^s$-MHD system \eqref{eq:mhd_angular_Powell} and so the rest of the proof to Theorem~\ref{thm:min_entropy}, and so Theorem~\ref{thm:generalized_entropy}. Theorem~\ref{thm:positivity_internal_energy} holds since both Theorems~\ref{thm:positivity_density} and \ref{thm:min_entropy} hold.
\end{proof}

\begin{theorem} \label{thm:conservation}
The GP-MHD and $\text{GP}^s$-MHD system satisfy the following conservation properties:

\begin{enumerate}
\item Both systems conserve mass;
\item Both systems conserve magnetism if $\alpha_\bB = 0$;
\item Both systems conserve momentum if $\alpha_\bbm = 0$;
\item The GP-MHD system conserves total energy if $\alpha_E = 0$ but the $\text{GP}^s$-MHD system does not;
\item The $\text{GP}^s$-MHD system conserves angular momentum if $\alpha_\bbm = 0$ but the GP-MHD system does not.
\end{enumerate}

\end{theorem}



\begin{proof}
We divide the proof of this theorem into several parts. We follow a similar proof technique as \cgreen{\cite{Charnyi2017,lundgren2023fully}}. We denote the $L^2 \l( \mR^d \r) $ inner product as $(\cdot,\cdot)$.

\textbf{(i) Conservation of mass.}
The mass equation is the same for the GP-MHD and $\text{GP}^s$-MHD systems. Integrating the mass equation and using the divergence theorem, we obtain
\[
\partial_t \int_{\mR^d} \rho \ud \bx = 0.
\]
This shows that mass is conserved.

\textbf{(ii) Conservation of magnetism.}
The GP-MHD and $\text{GP}^s$-MHD system have the same magnetic equation. Taking the inner product of the magnetic equation with a constant unit vector $\be_i$ yields
\begin{align*}
  (\partial_t \bB , \be_i) + ( \DIV (  \bu \otimes \bB ) , \be_i) - ( \DIV (  \bB \otimes \bu ) , \be_i) - \l( \DIV \l( \eta \l( \nabla \bB - \nabla \bB^\top \r) \r) , \be_i  \r) \\
  = \alpha_\bB ((\DIV \bB) \bu, \be_i).
\end{align*}
Performing integration by parts on some of the terms yields
\begin{equation*}
(\partial_t \bB , \be_i) - (   \bu \otimes \bB  , \nabla \be_i) + (   \bB \otimes \bu  , \nabla \be_i) + \l( \eta \l( \nabla \bB - \nabla \bB^\top \r) , \nabla \be_i  \r) = \alpha_\bB ((\DIV \bB) \bu , \be_i).
\end{equation*}
Since $\be_i$ is a constant vector we obtain $ \partial_t \int_{\mR^d} \bB_i \ud \bx =  (\partial_t \bB , \be_i) = \alpha_\bB ((\DIV \bB) \bu, \be_i)$ which shows that magnetism is conserved if $\alpha_\bB = 0$.

\textbf{(iii) Conservation of momentum.}
We start with the GP-MHD system. Taking the inner product of the momentum equations with a constant unit vector $\be_i$ yields
\begin{align*}
(\partial_t \bbm, \be_i) - ( \bbm \otimes \bu  , \GRAD \be_i) - ( p, \DIV \be_i) + (\bB \otimes \bB , \GRAD \be_i)&\\ - \frac{1}{2} \l( |\bB|^2, \DIV \be_i \r) + \l(  \mu \rho \nabla^s \bu + (\kappa \nabla \rho) \otimes \bu , \GRAD \be_i  \r) & = \alpha_\bbm ( (\DIV \bB) \bB, \be_i  ).
\end{align*}
Since $\be_i$ is a constant vector we obtain $\partial_t \int_{\mR^d} \bbm_i \ud \bx = (\partial_t \bbm , \be_i) = \alpha_\bbm ( (\DIV \bB) \bB, \be_i  )$, which means that momentum is conserved if $\alpha_\bbm = 0$. The same steps can be repeated for the $\text{GP}^s$-MHD system.

\textbf{(iv) Conservation of total energy.}
We integrate the energy equation and \cgreen{use} the divergence theorem to obtain
\[
\int_{\mR^d}\p_t E \ud x = \alpha_E \int_{\mR^d}(\DIV\bB)(\bB\SCAL\bu) \ud x.
\]
This shows that total energy is conserved if $\alpha_E = 0$. Repeating the same steps for the $\text{GP}^s$-MHD system yields
\begin{align*}
  \int_{\mR^d}\p_tE\ud x & = \alpha_E \int_{\mR^d}(\DIV\bB)(\bB\SCAL\bu) \ud x \\
  & \quad+ \int_{\mR^d} \l(  \DIV \l(  \frac{1}{2} ( \bu \otimes (\kappa\nabla\rho) - (\kappa\nabla\rho) \otimes \bu ) \r) \r) \SCAL \bu \ud x \neq 0.
\end{align*}
This shows that total energy is not conserved for the $\text{GP}^s$-MHD system.

\textbf{(v) Conservation of angular momentum.}
We define $\bphi_i \coloneqq  \bx \CROSS \be_i $ so that angular momentum can be expressed as $\int_{\mR^d} (\bbm \CROSS \bx)_i \ud \bx = (\bbm, \bphi_i) $. We note that $\bphi_i$ has the property that $\DIV \bphi = 0$ and $ \GRAD \bphi_i + (\GRAD \bphi_i)^\top = 0 $.  We start by considering angular momentum conservation of the GP-MHD system. Taking the inner product of the momentum equations with a constant unit vector $\bphi_i$ and integrating over the domain yields 
\begin{equation*}
\begin{split}
(\partial_t \bbm, \bphi_i) + (\DIV ( \bbm \otimes \bu  ), \bphi_i) + (\GRAD p, \bphi_i) - (\DIV ( \bB \otimes \bB ), \bphi_i) + \frac{1}{2} (\GRAD |\bB|^2, \bphi_i) \\ - \l( \DIV \l( \mu \rho \nabla^s \bu + (\kappa \nabla \rho) \otimes \bu \r), \bphi_i  \r) = \alpha_\bbm ( (\DIV \bB) \bB, \bphi_i  ).
\end{split}
\end{equation*}
Expanding the convection operators yields
\begin{equation*}
\begin{split}
  (\partial_t \bbm, \bphi_i) + ( \bu \SCAL \nabla \bbm + \bbm (\DIV \bu)    , \bphi_i) - \l(   \bB \SCAL \nabla \bB + (\DIV \bB) \bB, \bphi_i \r)  \\
  + \l( \mu \rho \nabla^s \bu + (\kappa \nabla \rho) \otimes \bu , \nabla  \bphi_i  \r) = \alpha_\bbm ( (\DIV \bB) \bB, \bphi_i  ),
\end{split}
\end{equation*}
since $\DIV \bphi_i = 0$. Integration by parts on $\bB \SCAL \nabla \bB$ and $\bu \SCAL \nabla \bbm$ yields
\begin{equation} \label{eq:step_ang}
\begin{split}
  (\partial_t \bbm, \bphi_i)  - ( \bu \SCAL \nabla \bphi_i, \bbm) - (\bbm (\DIV \bu)    , \bphi_i) + (\bbm (\DIV \bu)    , \bphi_i)  + (   \bB \SCAL \nabla \bphi_i, \bB ) \\
  +( (\DIV \bB) \bB, \bphi_i ) - ( (\DIV \bB) \bB, \bphi_i )
  + \l( \mu \rho \nabla^s \bu + (\kappa \nabla \rho) \otimes \bu , \nabla  \bphi_i  \r) = \alpha_\bbm ( (\DIV \bB) \bB, \bphi_i  ).
\end{split}
\end{equation}
By using \eqref{eq:trilinear definition} it can be shown that
\begin{equation*}
\begin{split}
(   \bB \SCAL \nabla \bphi_i, \bB ) &= \frac{1}{2}(   \bB \SCAL \nabla \bphi_i, \bB ) + \frac{1}{2}(   \bB \SCAL \nabla \bphi_i, \bB ) = \frac{1}{2} ( (\GRAD \bphi_i)^\top \bB, \bB ) + \frac{1}{2} ( (\GRAD \bphi_i)^\top \bB, \bB ) \\ & = \frac{1}{2} ( (\GRAD \bphi_i)^\top \bB, \bB ) + \frac{1}{2} ( (\GRAD \bphi_i) \bB, \bB ) = 0,
\end{split}
\end{equation*}
since $(\GRAD \bphi_i)^\top + \GRAD \bphi_i = 0.$ The same procedure can be used to show that  $(   \bu \SCAL \nabla \bphi_i, \bbm ) = (   \bu \SCAL \nabla \bphi_i, \bu \rho ) = 0$. Inserting this into \eqref{eq:step_ang} gives
\begin{equation}
\begin{split}
  (\partial_t \bbm, \bphi_i) + \l( \mu \rho \nabla^s \bu, \bphi_i \r) + \l(  (\kappa \nabla \rho) \otimes \bu , \nabla  \bphi_i  \r) = 0.
\end{split}
\end{equation}
Using that the contraction between a symmetric and anti-symmetric matrix with zero diagonal is zero \cite[Ch 11.2.1]{Larson_2013}, one can show that
\begin{equation} \label{eq:symmetry_trick}
\l(\frac{1}{2} \mu \rho  \l( \nabla \bu +  (\nabla \bu)^\top \r) , \nabla \bphi_i \r) = \frac{1}{2}  \l( \frac{1}{2} \mu \rho \l( \nabla \bu + \l( \nabla \bu \r)^\top \r) , \nabla \bphi_i + (\nabla \bphi_i )^\top \r).
\end{equation}
Since $\nabla \bphi_i = - \nabla \bphi_i^\top$ we are left with
\begin{equation*}
\begin{split}
  (\partial_t \bbm, \bphi_i) =\alpha_\bbm ( (\DIV \bB) \bB, \bphi_i  ) - \l(  (\kappa \nabla \rho) \otimes \bu , \nabla  \bphi_i  \r) \neq 0,
\end{split}
\end{equation*}
even if $\alpha_\bbm = 0$. Thus the GP-MHD system does not conserve angular momentum.

We now consider the $\text{GP}^s$-MHD system. Repeating the same steps as before we are instead left with
\begin{equation*}
\begin{split}
  (\partial_t \bbm, \bphi_i) = \alpha_\bbm ( (\DIV \bB) \bB, \bphi_i  ) - \frac{1}{2}  \l( (\kappa \nabla \rho) \otimes \bu  + \l( (\kappa\nabla\rho)\otimes\bu \r)^\top, \nabla  \bphi_i  \r) .
\end{split}
\end{equation*}
Next, we use the same trick as performed in \eqref{eq:symmetry_trick} to obtain
\begin{equation*} 
\begin{split}
  (\partial_t \bbm, \bphi_i)
  &= \alpha_\bbm ( (\DIV \bB) \bB, \bphi_i  ) - \frac{1}{4} \l( (\kappa \nabla \rho) \otimes \bu  + \l( (\kappa\nabla\rho)\otimes\bu \r)^\top, \nabla  \bphi_i + (\nabla  \bphi_i)^\top  \r) \\
  &= \alpha_\bbm ( (\DIV \bB) \bB, \bphi_i  ),
\end{split}
\end{equation*}
since $\nabla \bphi_i = - \nabla \bphi_i^\top$. Thus, angular momentum is conserved for the $\text{GP}^s$-MHD system if $\alpha_\bbm = 0$. \cblue{A numerical validation of this property can be found in Section~\ref{sec:numerical_results:angular_momentum}}.
\end{proof}

\section{Numerical investigation}\label{sec:numerical_results}
In this section, we demonstrate \cgreen{several} numerical properties of the GLM-GP-MHD system \eqref{eq:mhd_GP_Powell_GLM} and GLM-$\text{GP}^s$-MHD sytem \eqref{eq:mhd_angular_Powell}. The continuous Lagrange finite elements are used to discretize space and a fourth-order SSP explicit Runge-Kutta method \cite{Ruuth_2006} is used to discretize time.
\subsection{Finite element discretization}
Instead of $\mR^d$ in the continuous analysis, the computational domain is an open bounded domain $\Omega\subset\mR^d$. We consider a finite partition of the bounded polytope $\overline{\Omega}\approx\Omega$ into a finite number of disjoint simplex elements $K$, $\overline{\Omega}=\uplus K$ such that no vertex of any element is hanged on an edge of another element. At a fixed time $t$, we aim to find the finite element solution $\bsfU_h(t)\coloneqq(\rho_h,\bbm_h,E_h,\bB_h)^\top$ in $\bcalW_h$, $\bcalW_h\coloneqq\calQ_h\CROSS\bcalV_h\CROSS\calQ_h\CROSS\bcalV_h$, $\bcalV_h=[\calQ_h]^d$, and
\[
\calQ_h\coloneqq\{v(x) : v\in\calC^0(\overline{\Omega}), v|_{K}\in\polP_k,\;\forall K\},
\]
where $\polP_k$ is the space of Lagrange polynomials of $k$ degrees. A Galerkin finite element formulation for the GP-MHD system \eqref{eq:mhd_GP_Powell} reads: find $\bsfU_h(t)\in\calC^1([0,t]; \bcalW_h)$ such that
\begin{align*}
& \left(\p_t\bsfU_h,\bsfV_h\right)
+\left(\DIV(\bsfF_{\calE}(\bsfU_h)+\bsfF_{\calB}(\bsfU_h)),\bsfV_h\right)\\
& =
-\left(\bsfF_{\calV}^{\text{GP}}(\bsfU_h),\GRAD\bsfV_h\right) + \left(\bn\SCAL\bsfF_{\calV}(\bsfU_h),\bsfV_h\right)_{\p\Omega} + \left(\Psi,\bsfV_h\right),
\end{align*}
for all $\bsfV_h\in\bcalW_h$. Derivation for the other extensions such as GLM-GP-MHD, $\text{GP}^s$-MHD is done similarly. For the numerical experiments, we consider an ideal gas with the following equation of state,
\begin{equation}\label{eq:equation_of_state_ideal}
p = \rho T = (\gamma-1)\rho e.
\end{equation}

\subsection{Residual-based viscosity (RV)}\label{sec:RV}
We employ the residual viscosity method \cite{Nazarov_2013} for our finite element discretization where the proposed viscous flux works as a stabilization term. The RV method is used to scale the viscosity coefficients such that the discretization is stable while still being high-order accurate in space. Several benefits of the residual viscosity method are: (i) provable convergence to the unique entropy solution using implicit time-stepping; (ii) arbitrary high-order accuracy for smooth solutions \cite{Dao2022b}; (iii) being simple to construct and implement.

At every nodal point $i$, the viscosity coefficient is computed as
\[
\e_{h,i} \coloneqq
   \min\left(
  \frac12 h_i|\lambda_{\max,i}|,C_Eh_i^2\
   |R_i |
  \right),
\]
where $h_i$ is a local mesh size indicator, that is a nodal value of the mesh function $h_h \in \calQ_h$ computed by the following projection problem: Find $h_h \in \calQ_h$ such that 
\[
(h_h,v) + \sum_{K\in \calT_h}(|K|^{2/d} \GRAD h_h, v)_K = (k^{-1}|K|^{1/d}, v), \quad \forall v \in \calQ_h
\]
where $|K|$ is the \cgreen{volume} of the element $K$, $k$ is the polynomial order. Moreover $\lambda_{\max,i}$ is the exact or an estimated upper bound of the local maximum wave speed, \cblue{$C_E$ is a scaling number}, and $R_i$ is the nodal value of the PDE residual. Unless stated, $C_E$ is set to 1.0. See \cite{Dao2022b} for technical details and experimental results of the RV method. The different viscosities are set as $\kappa=\mu=\eta=\e_h$, see the unit analysis in \cite{Dao2022b}.
\subsection{1D contact wave problem}\label{sec:contact_line}
Contact discontinuities are an interesting phenomenon in gas dynamics. A contact wave or a contact line is a solution to \eqref{eq:mhd1} in which the velocity $\bu(\bx,t) = \bu_0$, the magnetic field $\bB(\bx,t) = \bB_0$, and the pressure $p(\bx,t)=p_0$ are constant but the density is a discontinuity. This section demonstrates some advantages of the GP flux over the more commonly used resistive MHD flux. The resistive MHD flux is defined as
\begin{equation}\label{eq:resistive_mhd_flux}
\bsfF_{\calV}^{\text{r}}(\bsfU)\coloneqq
\begin{pmatrix}
0 \\
2\mu\GRAD^s\bu+\lambda\DIV\bu\polI \\
(2\mu\GRAD^s\bu+\lambda\DIV\bu\polI) \bu + \kappa_T\GRAD T + \eta \big( \GRAD \bB - \GRAD \bB^\top \big)  \bB\\
\eta \big( \GRAD \bB - \GRAD \bB^\top \big)
\end{pmatrix}.
\end{equation}
Assume that the contact wave is a solution of \eqref{eq:mhd1} regularized by the resistive MHD flux. The density is a discontinuous function which satisfies the mass equation
\[
\p_t\rho + \DIV(\rho\bu_0) = 0.
\]
Inserting $\bu_0,\bB_0,p_0,\rho$ into the energy equation gives
\begin{equation}\label{eq:resistive_mhd_flux_contact_E}
\p_t(\rho e)+\DIV(\bu_0\rho e)-\DIV(\kappa_T\GRAD T) = 0.
\end{equation}
Consider the simplest case of ideal gases, due to \eqref{eq:equation_of_state_ideal}, if the pressure is a constant, then $\rho e$ is also a constant. This observation shows that \eqref{eq:resistive_mhd_flux_contact_E} holds only if the thermal diffusivity $\kappa_T\equiv 0$. Apart from having no physical meaning, zero thermal diffusivity leads to spurious oscillations of the numerical solutions, see \cite{Dao2022a,Nazarov_Larcher_2017}.

We repeat the above analysis on the GP flux. The density solution instead satisfies
\[
\p_t\rho + \DIV(\rho\bu_0) = \DIV(\kappa\GRAD\rho).
\]
Again, we insert $\bu_0,\bB_0,p_0,\rho$ into the energy equation to obtain
\[
\p_t(\rho e)+\DIV(\bu_0\rho e)-\DIV(\kappa\GRAD(\rho e))+\frac12\bu_0^2(\p_t\rho+\DIV(\rho\bu_0)-\DIV(\kappa\GRAD\rho))=0
\]
which holds trivially when substituting the ideal equation of state \eqref{eq:equation_of_state_ideal}.

We experiment with this example numerically. A contact wave profile is taken from \cite{Dao2022a}. Consider the domain $\Omega = [0, 1]$. The gas constant and initial components are
\begin{align*}
  \gamma&=2.0,\\
  \bu_0&=(0.5915470932,-1.5792628803)^\top,\\
  p_0&=0.5122334291,\\
  \bB_0&=(0.75,-0.5349102426)^\top.
\end{align*}
The density is initially set to be $\rho_{0,\text{L}}=0.7156521382$ on the left half and $\rho_{0,\text{R}}=0.2348529760$ on the right half of $\Omega$. With an aim to get bound-preserving solutions, we use first order viscosity $\e_{h}$ with the nodal values $\e_i=\frac12h_i|\lambda_{\max,i}|$ for all the viscosity coefficients. For the resistive MHD flux, the bulk viscosity coefficient $\lambda$ is set to be zero. The lumped mass matrix is used because it is impossible to get bound-preserving solutions with the consistent mass matrix \cite{Guermond_Popov_Yang_2017}. \cblue{For this reason, $\polP_2$ solutions are excluded in this test because of the zero entries in the resulting lumped mass.} \cblue{The numerical solutions at time $t=0.1$ is plotted in Figure~\ref{fig:contact_rho}. The meshes are chosen such that the $\polP_1$ and $\polP_3$ nodes exactly match: 61 DOFs using 60 $\polP_1$ elements or 20 $\polP_3$ elements, and 601 DOFs using 600 $\polP_1$ elements or 200 $\polP_3$ elements. Note that, since the initial discontinuous solutions are interpolated, there are bound violations in between the nodal points at the initial stage for the $\polP_3$ solutions. However, at the final time, the $\polP_3$ solutions exactly satisfy the maximum principle. The general behavior of $\polP_1$ and $\polP_3$ solutions do not differ significantly. Solutions by both fluxes converge toward the reference solution. The solutions obtained by the resistive MHD flux violate both the lower and upper bounds whereas the GP/GP$^s$ flux can capture the contact line without violating bounds in all resolutions.} The discrete minimum entropy $\min_{\Omega}s_h$ is plotted in Figure~\ref{fig:contact_min_s}. The finite element solutions using the GP flux satisfy the minimum entropy principle to machine precision. However, the resistive MHD flux does not result in the same behavior.

     

\begin{figure}[!ht]
  \centering
  \includegraphics[width=0.49\textwidth]{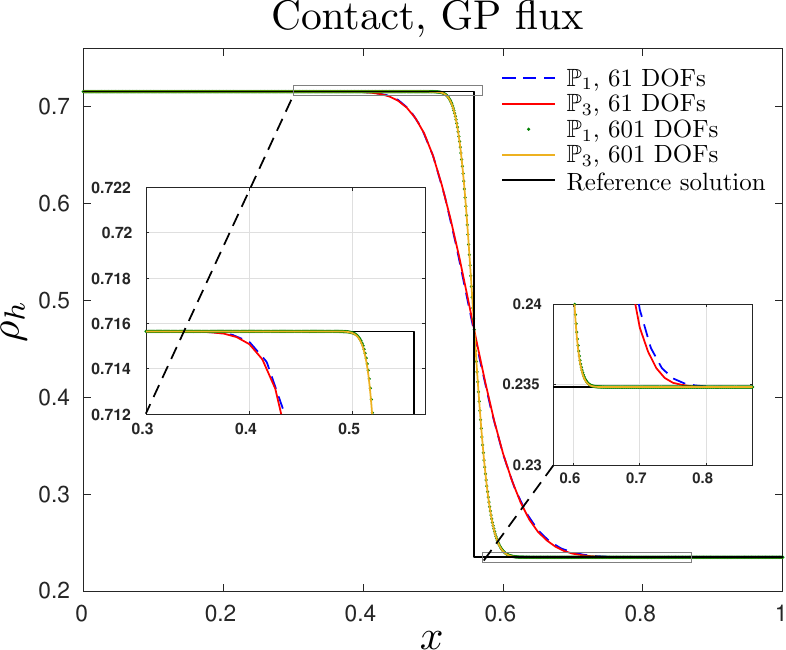}
  \includegraphics[width=0.49\textwidth]{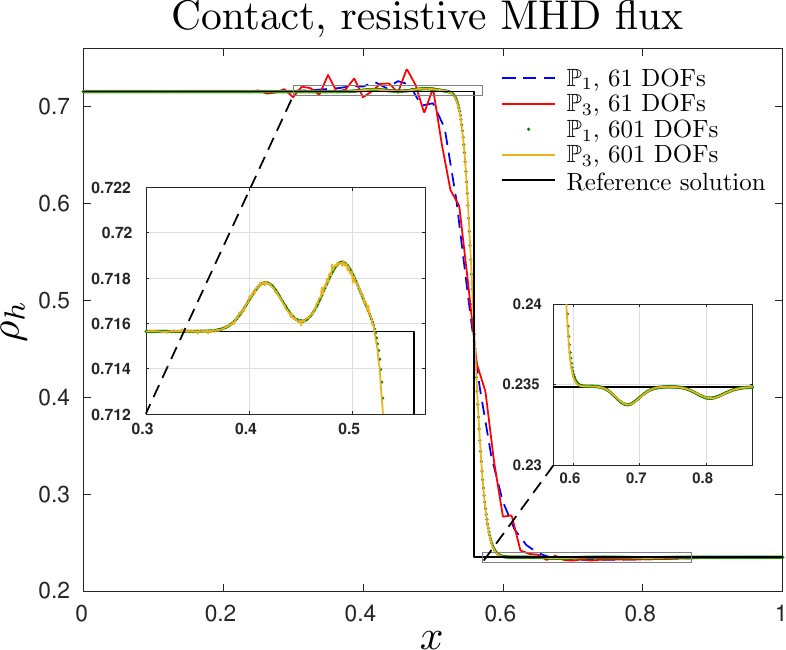}
  \caption{\cblue{The resistive MHD flux violates the maximum principle around the location of the contact line. First order viscosity, $\polP_1$ and $\polP_3$ elements. Left: GP flux; $\polP_1$ and $\polP_3$ solutions preserve bounds exactly. Right: resistive MHD flux; for visibility only the solutions with 601 DOFs are shown in the zoomed-in plots. Note that GP and GP$^s$ are same in 1D.}}
  \label{fig:contact_rho}
\end{figure}

\begin{figure}[h!]
     \centering
     \begin{subfigure}{0.49\textwidth}
         \centering
         \includegraphics[width=\textwidth,viewport=95 448 486 778, clip=true]{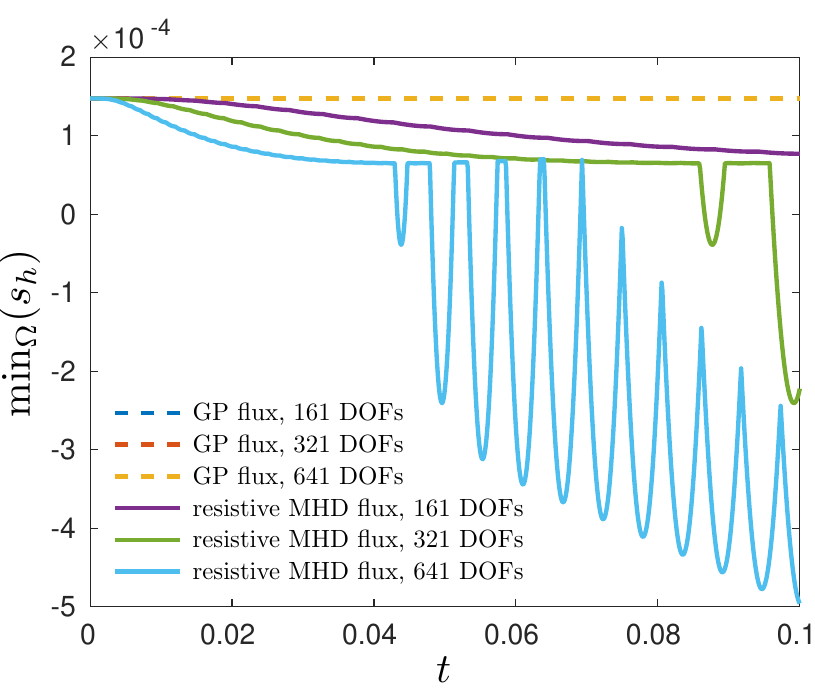}
     \end{subfigure}
     \caption{Investigation of discrete minimum entropy principle on the contact line problem. First order viscosity, $\polP_1$ elements. The GP/GP$^s$ flux preserves $\min_{\Omega}(s_h)$ to machine precision. The resistive MHD flux produces unphysical behaviors in $\min_{\Omega}(s_h)$.}
     \label{fig:contact_min_s}
\end{figure}

\subsection{Accuracy test}
We demonstrate the high-order accuracy of the RV method using the smooth vortex problem \cite{Wu_2018b}. We describe our setup in \cite{Dao2022a,Dao2022b}. The convergence results are presented in Table~\ref{table:convergence_vortex_P1} for $\polP_1$, $\polP_2$ and $\polP_3$ elements. Second order accuracy is observed in the case of $\polP_1$ elements. Fourth order is obtained in the case of $\polP_3$ elements. However, we also obtain second order solutions using $\polP_2$ elements, which might be considered suboptimal with respect to the polynomial degree. It turns out that this numerical behavior is expected for the continuous Galerkin finite elements and is explained in \cite{Ainsworth_2014,Dao2022b}. The errors between GP and GP$^s$ fluxes are identical in the tables. The reason is that the discrepancy between the two are small, and the viscosity coefficients are of order $\calO(h^{p+1})$ for smooth solutions. To compare the GP and GP$^s$ fluxes in presence of shocks, we include the Orszag-Tang benchmark in Section~\ref{sec:OT}. 

\subsection{Brio-Wu problem}
In this section, the popular 1D MHD benchmark by \cite{Brio_Wu_1988} is experimented. We aim to test the convergence of the method in presence of discontinuities. Since an exact solution to this problem is not known, and since different exact Riemann solvers do not agree on the existence of intermediate shocks, we employ the numerical solution of the Athena code \cite{Stone_2008} using 10001 grid points as the reference solution. Solutions at different levels using $\polP_1$ polynomials are compared in Figure~\ref{fig:briowu}. The zoomed-in plots show the convergence behavior of the numerical method. The convergence rates are reported in Table~\ref{table:convergence_briowu}. The obtained errors converge to zero at first order in L$^1$-norm and half an order in L$^2$. \cblue{In the Brio-Wu tests, we set the high-order viscosity scaling number $C_E$ to be 5.0 to reduce the numerical oscillations for better illustrations at low resolution. We note that letting $C_E=1.0$ does not break convergence, and tuning $C_E$ does not change the convergence rates. Nevertheless, ensuring absolute no oscillations requires the use of a positivity-preserving scheme which respects local bounds of every nodal update for fully discrete scheme, see e.g., \cite{dao2023structure}, which is not the focus of this paper. A comparison of the GP flux and the resistive MHD flux for this problem for different polynomial degrees is shown in Section~\ref{sec:briowu_GP_vs_resistive}. 
}

\begin{figure}[h!]
     \centering
     \begin{subfigure}{0.49\textwidth}
         \centering
         \includegraphics[width=\textwidth]{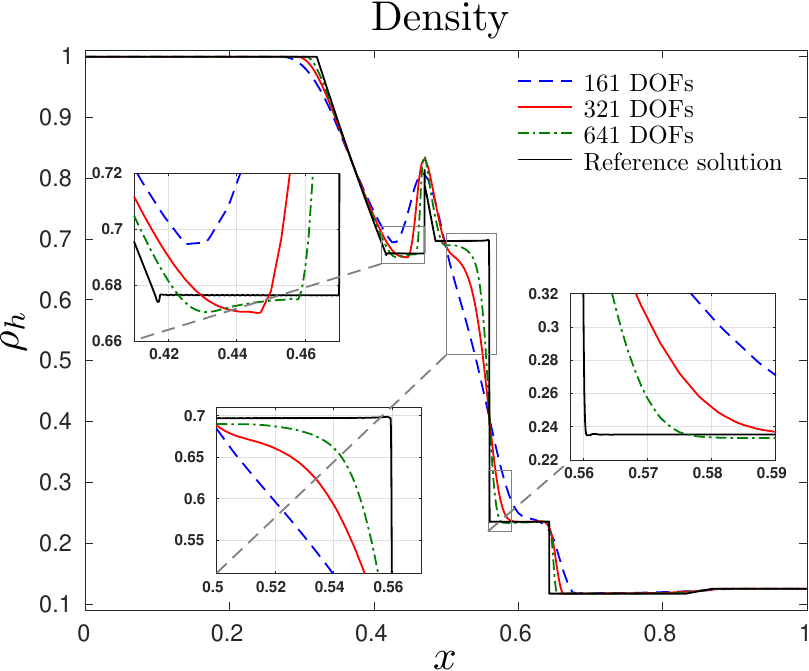}
     \end{subfigure}
     \hfill
     \begin{subfigure}{0.49\textwidth}
         \centering
         \includegraphics[width=\textwidth]{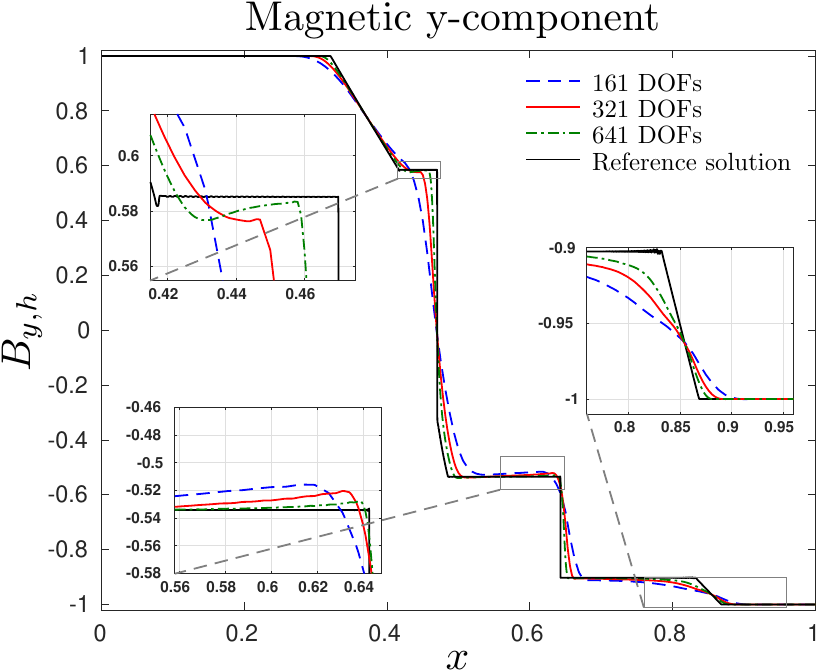}
     \end{subfigure}
     \caption{Solution to the Brio-Wu problem using $\polP_1$ polynomials, GP/GP$^s$ flux, and residual viscosity. End time $\widehat t=0.1$.}
     \label{fig:briowu}
\end{figure}

\subsection{GEM magnetic reconnection challenge}
Magnetic reconnection is an interesting phenomenon in MHD and plays an important role in understanding many processes in astrophysics and plasma physics. To compare different MHD models, a test problem was proposed by \cite{Birn_2001}. Our aim is to check if the magnetic reconnection behaviors of the new viscous models are similar to the resistive MHD model. Only the results by the GP flux is shown in this test because the difference to the GP$^s$ flux is negligible. The computational domain is a rectangle $[-L_x/2,L_x/2]\times[-L_y/2,L_y/2]$ with the lengths $L_x=25.6,L_y=12.8$. The initial density profile is
\[
\rho_0 = \frac{1}{\cosh(2x)\cosh(2y)}+0.2.
\]
The initial velocity is zero, $\bu_0=(0,0)^\top$. The initial pressure is set as $p_0=\frac12\rho$. The initial magnetic field is,
\begin{align*}
  \bB &= (B_x,B_y),\\
  B_x &= \tanh(2y)+\delta B_x,\\
  B_y &= \delta B_y,\\
  \delta B_x &= \frac{-0.1\pi}{L_y}\sin\left(\frac{\pi y}{L_y}\right)\cos\left(\frac{2\pi x}{L_x}\right),\\
  \delta B_y &= \frac{0.2\pi}{L_x}\sin\left(\frac{2\pi x}{L_x}\right)\cos\left(\frac{\pi y}{L_y}\right).
\end{align*}
The domain is periodic in the $x$-direction. Slip boundary condition is strongly imposed on the top and the bottom boundaries. The magnetic reconnection is measured by \cite{Rueda-Ramirez_2021,Birn_2001},
\[
\bbf_{\mathrm{rec}}(t)=\frac{1}{2} \int_{-L_x / 2}^{L_x / 2}\left|B_y(x, y=0, t)\right| \ud x.
\]
The physical viscosities are $\kappa_{\mathrm{physical}}=0$, $\mu_{\mathrm{physical}} = 0$, $\eta_{\mathrm{physical}} = 5\times10^{-3}$. To incorporate the artificial viscosity, for every node $i$ we simply choose numerical $\eta_i$ to be
\[
\eta_i = \max(\eta_{\mathrm{physical}}, \eta_{\mathrm{RV},i}),
\]
where $\eta_{\mathrm{RV},i}$ is the nodal value of the artificial viscosity by the RV method. We use $1000\times1000$ $\polP_1$ nodes for the solution plots in Figure~\ref{fig:GEM_rho_B}. Even with the physical viscosities, for the given resolution, the time step restriction from the advective flux is still dominant. Therefore, the explicit RK time stepping is used without any changes. The magnetic reconnection is measured over time and is plotted in Figure~\ref{fig:GEM_reconnection}. From the figure, we can see that the reconnection rates are very similar between the two viscous models: resistive MHD and GP. Moreover, our result closely matches the resistive MHD results \cred{produced by the FV/DGSEM scheme \cite{Rueda-Ramirez_2021} with mesh size $1024\times512$, degree 7 polynomials which has higher resolution than our setup.}

In addition, we investigate the case that the physical resistivity is set to zero to which we expect no reconnection to happen. Indeed, the reconnection remains unchanged for a period of time at the beginning of the simulation. However, this state is broken when the instabilities appear which drastically shoot up the reconnection measurement after $t = 28$. This behavior is shown by the blue dotted line in Figure~\ref{fig:GEM_reconnection}. \cblue{We omit the results of higher-order polynomials because under the same number of DOFs they produce nearly identical lines to the $\polP_1$ solution.}

\begin{figure}[!ht]
     \centering
     \begin{subfigure}{0.49\textwidth}
         \centering
         \includegraphics[width=\textwidth]{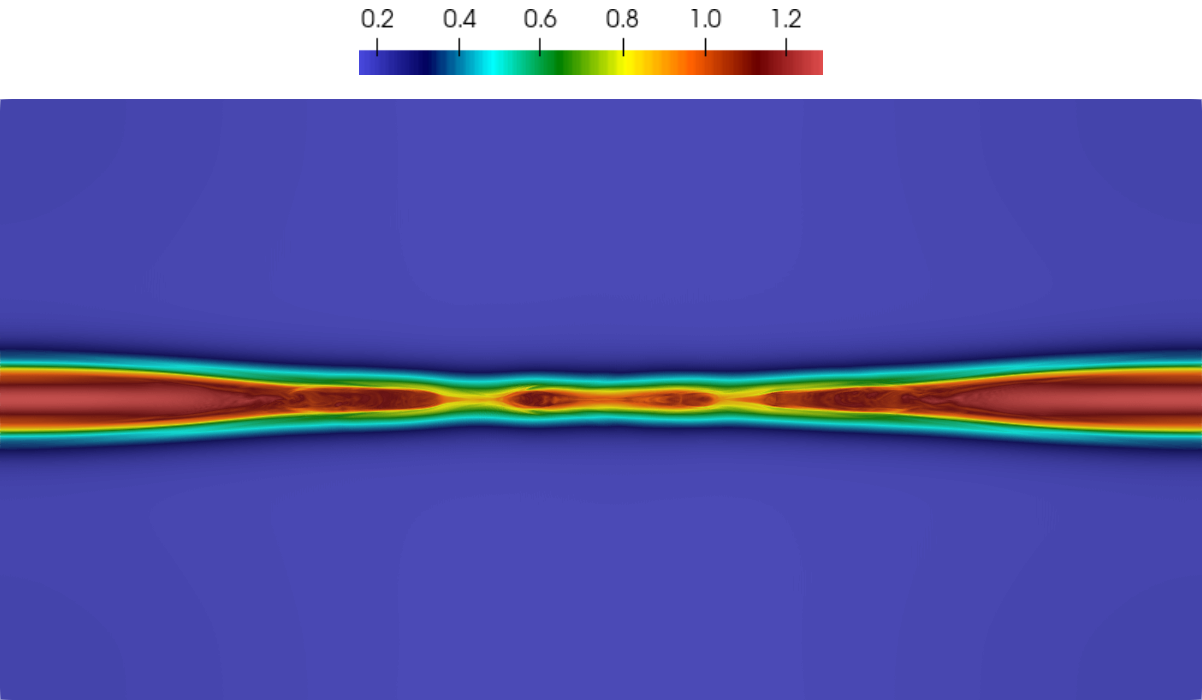}
         \caption{Density $\rho_h$, $\eta=0$}
     \end{subfigure}
     \hfill
     \begin{subfigure}{0.49\textwidth}
         \centering
         \includegraphics[width=\textwidth]{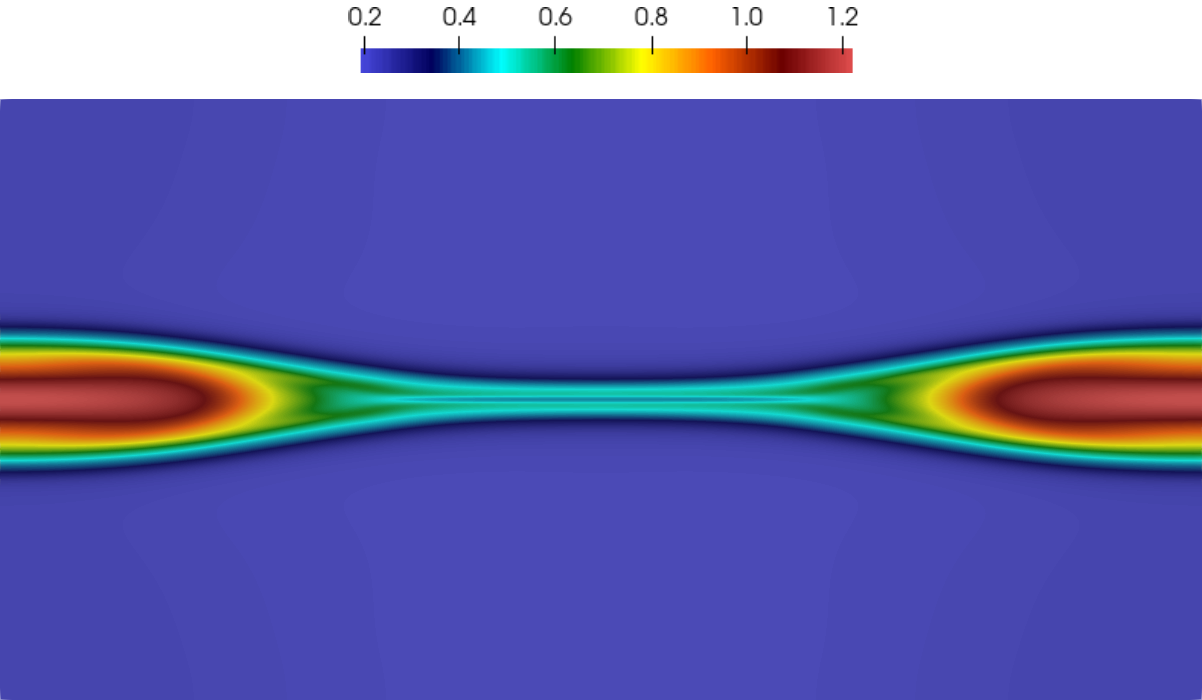}
         \caption{Density $\rho_h$, $\eta=5$E-3}
     \end{subfigure}
     \hfill
     \begin{subfigure}{0.49\textwidth}
         \centering
         \includegraphics[width=\textwidth]{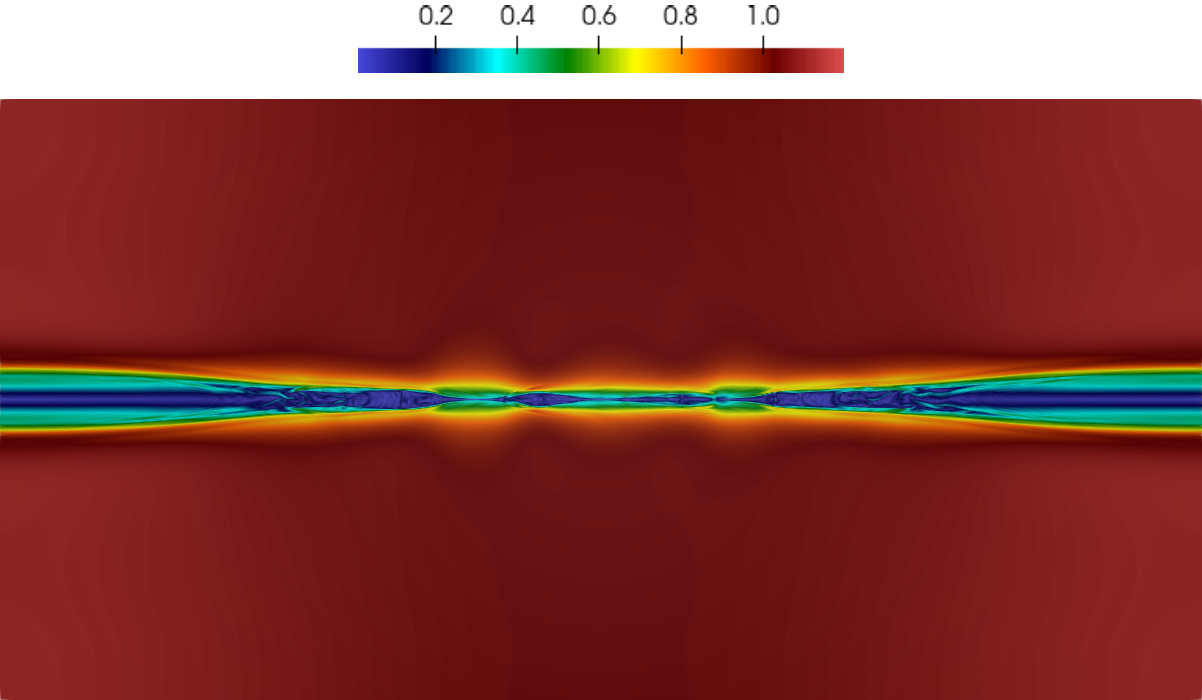}
         \caption{Magnetic strength $\vert\bB_h\vert$, $\eta=0$}
     \end{subfigure}
     \hfill
     \begin{subfigure}{0.49\textwidth}
         \centering
         \includegraphics[width=\textwidth]{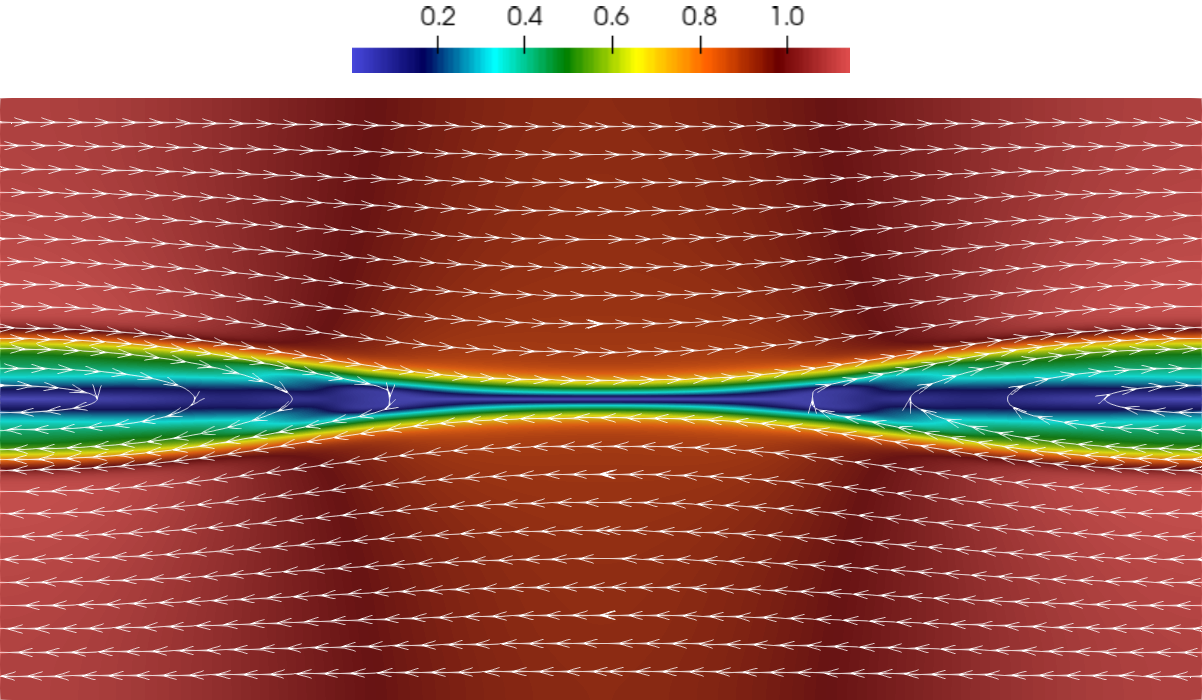}
         \caption{Magnetic strength $\vert\bB_h\vert$, $\eta=5$E-3}
     \end{subfigure}
     \caption{Solutions at time $t=40$ for the GEM challenge.}
     \label{fig:GEM_rho_B}
\end{figure}

\begin{figure}[!ht]
     \centering
     \includegraphics[width=0.6\textwidth,viewport=100 454 460 747, clip=true]{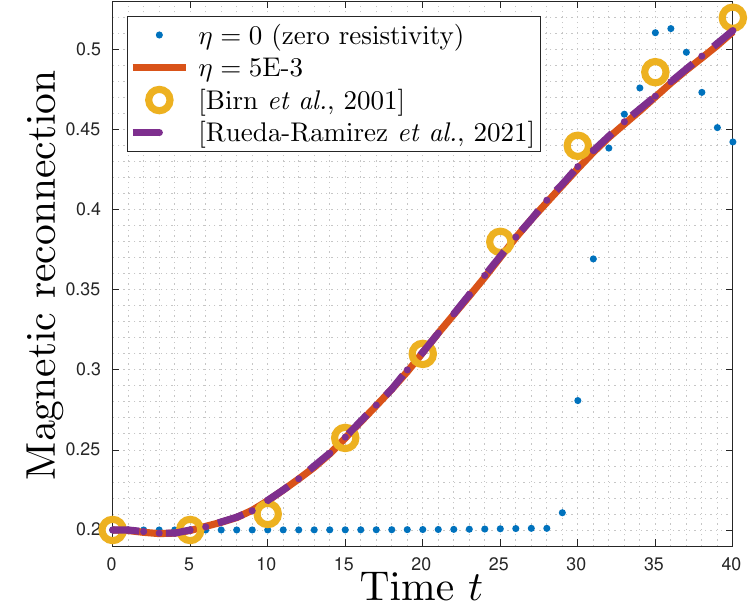}
     \caption{Magnetic reconnection rate for the GEM challenge with $\eta=0$ (zero resistivity) and $\eta=5$E$-3$. The results are compared with two other references.}
     \label{fig:GEM_reconnection}
\end{figure}






\section{Summary and conclusion}\label{sec:summary}

In this paper, we have investigated several options for viscous regularization of the ideal MHD equations. We summarize our findings in Tables~\ref{table:conclusion_viscous_model}, \ref{table:conclusion_divergence} and \ref{table:conclusion_GLM}. Table~\ref{table:conclusion_viscous_model} presents the properties of the different viscous models when the magnetic divergence is assumed to be zero. When nonzero divergence is considered, properties of different divergence source terms are presented in Table~\ref{table:conclusion_divergence} and properties of the GLM methods are compared in Table~\ref{table:conclusion_GLM}. In Table~\ref{table:conclusion_GLM}, $\checkmark^*$ refers to that the GLM method conserves $E^*$.

The proposed viscous regularization is numerically experimented with an artificial viscosity finite element method. The numerical results show that: (i) high-order accuracy is obtained for smooth solutions; (ii) shocks and other discontinuities are finely captured; (iii) the numerical behaviors are strongly aligned with the continuous analysis; (iv) the proposed viscous flux behaves similarly to the resistive MHD flux in a test with physical viscosity.

\section*{Acknowledgments}
Some computations were performed on UPPMAX provided by the Swedish National Infrastructure for Computing (SNIC) under project number SNIC 2021/22-233.


\newcommand \subWidth{17mm}

\begin{landscape}
\begin{table}
    \centering
    \caption{Summary: viscosity/resistivity models assuming $\DIV \bB = 0$. \cred{The conserved quantities are defined in Section~\ref{sec:conservativeness}}.}
    \label{table:conclusion_viscous_model}
    \begin{adjustbox}{max width=1.5\textwidth}
    \begin{tabular}{|c||c|c|c|c|c|c|c|c|c|c|}
    \hline
    \multirow{2}{*}{Equations} & \multirow{2}{*}{$\rho, e > 0$} & \multirow{2}{*}{Entropy ineq.}  & \multirow{2}{*}{Galilean inv.} & \multirow{2}{*}{Rotational inv.} & \multicolumn{6}{c|}{Conservativeness} \\ \cline{6-11}
    {} & {} & {} & {} & {} & $\rho$ & $\bbm$ & $\bB$ & $E$& $\bbm \CROSS \bx$ & $\bA \SCAL \bB$\\ \hline\hline
    Ideal MHD          &               &               & \checkmark & \checkmark & \checkmark & \checkmark & \checkmark & \checkmark & \checkmark & \checkmark \\ \hline
    Resistive MHD      &               &               & \checkmark & \checkmark & \checkmark & \checkmark & \checkmark & \checkmark & \checkmark &            \\ \hline
    Monolithic-MHD \cite{Dao2022b}     & \checkmark    & \checkmark    & \checkmark & \checkmark & \checkmark & \checkmark & \checkmark & \checkmark &            &            \\ \hline
    GP-MHD             & \checkmark    & \checkmark    & \checkmark & \checkmark & \checkmark & \checkmark & \checkmark & \checkmark &            &            \\ \hline
    $\text{GP}^s$-MHD  & \checkmark    & \checkmark    &            & \checkmark & \checkmark & \checkmark & \checkmark &            & \checkmark &             \\ \hline
    \end{tabular}
    \end{adjustbox}
    \caption{Summary: divergence source terms $\Psi$ when $\DIV \bB \neq 0$}
    \label{table:conclusion_divergence}
    \begin{adjustbox}{max width=1.5\textwidth}
    \begin{tabular}{|c||c|c|c|c|c|c|c|c|}
    \hline
    \multirow{2}{*}{Equations} & \multirow{2}{*}{Entropy ineq.}  & \multirow{2}{*}{Galilean inv.} & \multirow{2}{*}{Rotational inv.} & \multicolumn{5}{c|}{Conservativeness} \\ \cline{5-9}
    {} & {} & {} & {} & $\rho$ & $\bbm$ & $\bB$ & $E$& $\bbm \CROSS \bx$ \\ \hline\hline
    $\Psi(\alpha_{\bbm},\alpha_E,\alpha_{\bB})$  & $\alpha_E-\alpha_{\bbm}-\alpha_{\bB}=1$ & $\alpha_{\bbm}=\alpha_{E},\alpha_{\bB}=-1$  & \checkmark & \checkmark & $\alpha_{\bbm}=0$ & $\alpha_{\bB}=0$ & $\alpha_{E}=0$ & $\alpha_{\bbm}=0$ \\ \hline
    $\Psi_{\text{Powell}}=\Psi(-1,-1,-1)$  \cite{Powell_et_al_1999}     & \checkmark                            & \checkmark                               & \checkmark &    \checkmark  &  &  &   &      \\ \hline
    $\Psi_{\text{Janhunen}}=\Psi(0,0,-1)$  \cite{Janhunen_2000}      & \checkmark                            & \checkmark                               & \checkmark & \checkmark  & \checkmark &   & \checkmark &  \checkmark \\ \hline
    $\Psi_{\text{BB}}=\Psi(-1,0,0)$   \cite{Brackbill_Barnes_1980}     & \checkmark                            &                                    & \checkmark &  \checkmark  &  & \checkmark  &\checkmark & \\ \hline
     $\Psi(0,1,0)$                            & \checkmark                            &                                          & \checkmark & \checkmark  &  \checkmark &  \checkmark &  & \checkmark \\ \hline
    \end{tabular}
    \end{adjustbox}
    \caption{Summary: properties of the GLM}
    \label{table:conclusion_GLM}
    \begin{adjustbox}{max width=1.5\textwidth}
    \begin{tabular}{|c||c|c|c|c|c|c|c|c|c|}
    \hline
    & \multirow{2}{*}{Div. cleaning} & \multirow{2}{*}{Entropy ineq.}  & \multirow{2}{*}{Galilean inv.} & \multirow{2}{*}{Rotational inv.} & \multicolumn{5}{c|}{Conservativeness} \\ \cline{6-10}
        {} & {} & {} & {} & {} & $\rho$ & $\bbm$ & $\bB$ & $E$& $\bbm \CROSS \bx$ \\ \hline\hline
        No GLM & & \checkmark & \checkmark  & \checkmark &  \checkmark &  \checkmark  & \checkmark   &  \checkmark & \checkmark \\ \hline
        Galilean invariant extended GLM \cite{Dedner_et_al_2002} & \checkmark & \checkmark & \checkmark  & \checkmark & \checkmark  &  \checkmark  &  \checkmark   &   & \checkmark \\ \hline
        Galilean invariant GLM 9 waves \cite{Derigs_2018} & \checkmark & \checkmark & \checkmark  & \checkmark & \checkmark  &  \checkmark  &  \checkmark   &   & \checkmark \\ \hline
        Energy conservative GLM (Section~\ref{Sec:energy_GLM})  & \checkmark & \checkmark &   & \checkmark & \checkmark  &  \checkmark  &  \checkmark   & \checkmark*  & \checkmark \\ \hline
    \end{tabular}
    \end{adjustbox}
\end{table}
\end{landscape}

\newpage
\appendix

\section{Useful integration rules}\label{appendix:useful_rules}
Let $\bu, \bw, \bv \in \bH^1\l(\mR^d \r) = \l[ H^1 \l( \mR^d\r) \r]^d$ where $H^1\l( \mR^d\r)$ is a Hilbert space. Using the definition of the $L^2$-inner product $( \SCAL , \SCAL )$ and $( \GRAD \bu) \coloneqq \partial_{x_i} u_j $ one can show the following relation
\begin{equation} \label{eq:trilinear definition}
(\bu \SCAL \GRAD \bv, \bw) =\l( ( \GRAD \bv)^\top \bu, \bw\r) = ( (\GRAD \bv) \bw , \bu).
\end{equation}
The following relations hold due to integration by parts
\begin{align}
(\bu \SCAL \GRAD \bv, \bw) = -( (\DIV \bu) \bv, \bw  ) - ( \bu \SCAL \GRAD \bw, \bv ), \label{eq:IBP} \\
( \nabla \CROSS \bu, \bv ) = ( \bu, \nabla \CROSS \bv ). \label{eq:IBP cross}
\end{align}

\section{Proof of Galilean and rotational invariance of \eqref{eq:mhd_GP_Powell}}
\subsection{Galilean invariance}\label{appendix:galilean_invariance}
For readability, the analysis is done in 2D. Extension to 3D is trivial. Consider a shift of the time-space frame from $(t, x_1, x_2)$ to $(\tau,\xi_1,\xi_2)$. The inertial frame $(\tau,\xi_1,\xi_2)$ moves with a constant velocity of $V$, $|V| \ll c$ along the $x_1$-direction relatively to the reference frame $(t, x_1, x_2)$. A particle moving at velocity $\bu = (u_1, u_2)^\top$ in $(t, x_1, x_2)$ is observed as moving at velocity $\bv = (v_1, v_2)^\top$ in $(\tau,\xi_1,\xi_2)$,
\[
\begin{array}{lll}
  \tau = t, & \xi_1 = x_1 - Vt, & \xi_2 = x_2,\\
  & v_1 = u_1 - V, & v_2 = u_2.
\end{array}
\]
Of physical relevance, we consider the magnetic limit of Galilean electromagnetism, see \cite{Montigny_2006}. The displacement current is neglected in the non-relativistic MHD \eqref{eq:mhd1}, yielding that the magnetic field $\bB = (b_1, b_2)^\top$ stays invariant under the Galilean transformation.
By partial derivative rules, we have
\[
\begin{array}{lll}
\p_t = \p_\tau-V\p_{\xi_1}, & \p_{x_1} = \p_{\xi_1}, & \p_{x_2} = \p_{\xi_2}.
\end{array}
\]
To begin with, we consider the mass equation,
\[
\p_t\rho + \p_{x_1}(\rho u_1) + \p_{x_2}(\rho u_2) = \p_{x_1}(\kappa\p_{x_1}\rho)+\p_{x_2}(\kappa\p_{x_2}\rho)
\]
Rewrite the equation in the new reference frame,
\[
\p_{\tau}\rho-V\p_{\xi_1}\rho+\p_{\xi_1}(\rho (v_1+V)) + \p_{\xi_2}(\rho v_2) = \p_{\xi_1}(\kappa\p_{\xi_1}\rho)+\p_{\xi_2}(\kappa\p_{\xi_2}\rho),
\]
which can be simplified as
\[
\p_{\tau}\rho+\p_{\xi_1}(\rho v_1) + \p_{\xi_2}(\rho v_2) = \p_{\xi_1}(\kappa\p_{\xi_1}\rho)+\p_{\xi_2}(\kappa\p_{\xi_2}\rho).
\]
The mass equation is Galilean invariant because it stays the same in $(\tau,\xi_1,\xi_2)$.

We continue with the momentum equation,
\[
\p_t\begin{pmatrix}
\rho u_1\\
\rho u_2
\end{pmatrix}
+
\p_{x_1}\begin{pmatrix}
  \rho u_1^2+p\\
  \rho u_1u_2
\end{pmatrix}
+
\p_{x_2}\begin{pmatrix}
  \rho u_1u_2\\
  \rho u_2^2+p
\end{pmatrix}
+
\p_{x_1}\begin{pmatrix}
  \frac12 |\bB|^2-b_1^2\\
  -b_1 b_2
\end{pmatrix}
+
\p_{x_2}\begin{pmatrix}
  -b_1 b_2\\
  \frac12|\bB|^2 -b_2^2
\end{pmatrix}
=
\]
\[
\p_{x_1}\begin{pmatrix}
  \mu\rho\p_{x_1}u_1+\kappa(\p_{x_1}\rho)u_1\\
  \mu\rho\frac{\p_{x_1}u_2+\p_{x_2}u_1}{2}+\kappa(\p_{x_1}\rho)u_2
\end{pmatrix}
+
\p_{x_2}\begin{pmatrix}
  \mu\rho\frac{\p_{x_1}u_2+\p_{x_2}u_1}{2}+\kappa(\p_{x_2}\rho)u_1\\
  \mu\rho\p_{x_2}u_2+\kappa(\p_{x_2}\rho)u_2\\
\end{pmatrix}
\]
\[
+\begin{pmatrix}
-b_1(\p_{x_1}b_1+\p_{x_2}b_2)\\
-b_2(\p_{x_1}b_1+\p_{x_2}b_2)
\end{pmatrix}.
\]
The magnetic contribution stays invariant, meaning that
\[
\begin{array}{cc}
\p_{x_1}\begin{pmatrix}
  \frac12 |\bB|^2-b_1^2\\
  -b_1 b_2
\end{pmatrix}
=
\p_{\xi_1}\begin{pmatrix}
  \frac12 |\bB|^2-b_1^2\\
  -b_1 b_2
\end{pmatrix},
&
\p_{x_2}\begin{pmatrix}
  -b_1 b_2\\
  \frac12|\bB|^2 -b_2^2
\end{pmatrix}
=
\p_{\xi_2}\begin{pmatrix}
  -b_1 b_2\\
  \frac12|\bB|^2 -b_2^2
\end{pmatrix},
\end{array}
\]
\[
\begin{pmatrix}
-b_1(\p_{x_1}b_1+\p_{x_2}b_2)\\
-b_2(\p_{x_1}b_1+\p_{x_2}b_2)
\end{pmatrix}
=
\begin{pmatrix}
-b_1(\p_{\xi_1}b_1+\p_{\xi_2}b_2)\\
-b_2(\p_{\xi_1}b_1+\p_{\xi_2}b_2)
\end{pmatrix}.
\]
The left hand side of the first momentum equation without the magnetic terms reads
\[
\begin{array}{c}
  \p_t(\rho u_1) + \p_{x_1}(\rho u_1^2+p)+\p_{x_2}(\rho u_1 u_2)\\
  = \p_{\tau}(\rho(v_1+V))-V\p_{\xi_1}(\rho(v_1+V))+\p_{\xi_1}(\rho(v_1+V)^2+p)+\p_{\xi_2}(\rho(v_1+V)v_2)\\
  = \p_{\tau}(\rho v_1) + \p_{\xi_1}(\rho v_1^2)+\p_{\xi_2}(\rho v_1v_2) + V (\p_\tau\rho+\p_{\xi_1}(\rho v_1)+\p_{\xi_2}(\rho v_2))\\
  = \p_{\tau}(\rho v_1) + \p_{\xi_1}(\rho v_1^2)+\p_{\xi_2}(\rho v_1v_2) + V(\p_{\xi_1}(\kappa\p_{\xi_1}\rho)+\p_{\xi_2}(\kappa\p_{\xi_2}\rho)).
\end{array}
\]
The right hand side of the first momentum equation without the magnetic terms is
\[
\begin{array}{c}
\p_{x_1}(\mu\rho\p_{x_1}u_1+\kappa(\p_{x_1}\rho)u_1)
+
\p_{x_2}\left(\mu\rho\frac{\p_{x_1}u_2+\p_{x_2}u_1}{2}+\kappa(\p_{x_2}\rho)u_1\right)\\
=
\p_{\xi_1}(\mu\rho\p_{\xi_1}(v_1+V)+\kappa(\p_{\xi_1}\rho)(v_1+V))\\
+
\p_{\xi_2}\left(\mu\rho\frac{\p_{\xi_1}v_2+\p_{\xi_2}(v_1+V)}{2}+\kappa(\p_{\xi_2}\rho)(v_1+V)\right)\\
= \p_{\xi_1}(\mu\rho\p_{\xi_1}v_1+\kappa(\p_{\xi_1}\rho)v_1)
+
\p_{\xi_2}\left(\mu\rho\frac{\p_{\xi_1}v_2+\p_{\xi_2}v_1}{2}+\kappa(\p_{\xi_2}\rho)v_1\right) \\
+V(\p_{\xi_1}(\kappa\p_{\xi_1}\rho)+\p_{\xi_2}(\kappa\p_{\xi_2}\rho)).
\end{array}
\]
The term $V(\p_{\xi_1}(\kappa\p_{\xi_1}\rho)+\p_{\xi_2}(\kappa\p_{\xi_2}\rho))$ appearing on both sides of the equation are cancelled. We obtain the same equation in the new coordinates. Similar derivation can be done on the second momentum equation. Therefore, the momentum equation is also Galilean invariant.

For convenience, we split the energy equation into four parts, $\text{LHS}_1^{x_1,x_2}$ $+$ $\text{LHS}_2^{x_1,x_2}$ $=$ $\text{RHS}_1^{x_1,x_2}$ $+$ $\text{RHS}_2^{x_1,x_2}$, where
\begin{align*}
  \text{LHS}_1^{x_1,x_2} & = \p_t E + \p_{x_1}(u_1(E+p))+\p_{x_2}(u_2(E+p)),\\
  \text{LHS}_2^{x_1,x_2} & = \p_{x_1}\left(\left(\frac12|\bB|^2- b_1^2\right)u_1-b_1b_2u_2\right) + \p_{x_2}\left(-b_1b_2u_1+\left(\frac12|\bB|^2-b_2^2\right)u_2\right),\\
  \text{RHS}_1^{x_1,x_2} & = \p_{x_1}(\kappa\p_{x_1}(\rho e)) + \p_{x_2}(\kappa\p_{x_2}(\rho e))\\
  & + \p_{x_1}(\frac12|\bu|^2\kappa\p_{x_1}\rho) + \p_{x_2}(\frac12|\bu|^2\kappa\p_{x_2}\rho)\\
  & +\p_{x_1}(\mu\rho u_1\p_{x_1}u_1+\mu\rho\frac12(\p_{x_1}u_2+\p_{x_2}u_1)u_2) \\
  & +\p_{x_2}(\mu\rho \frac12(\p_{x_1}u_2+\p_{x_2}u_1)u_1+\mu\rho u_2\p_{x_2}u_2) \\
  & + \p_{x_1}(\eta b_2(\p_{x_1}b_2-\p_{x_2}b_1)) \\
  & + \p_{x_2}(\eta b_1(\p_{x_x}b_1-\p_{x_1}b_2)), \\
  \text{RHS}_2^{x_1,x_2} & = -(u_1b_1+u_2b_2)(\p_{x_1}b_1+\p_{x_2}b_2).
\end{align*}
Applying \eqref{eq:e} on $E$ gives
\[
\begin{array}{rl}
  \text{LHS}_1^{x_1,x_2} = & \text{LHS}_1^{\xi_1,\xi_2}\\
  &+ \frac12 V^2\left[\p_\tau\rho+\p_{\xi_1}(\rho v_1)+\p_{\xi_2}(\rho v_2)\right]\\
  &+ V \left[\p_{\tau}(\rho v_1)+\p_{\xi_1}(\rho v_1^2+p)+\p_{\xi_2}(\rho v_1 v_2)\right],\\
   \text{LHS}_2^{x_1,x_2} = & \text{LHS}_2^{\xi_1,\xi_2} +V\left[\p_{\xi_1}(\frac12|\bB|^2- b_1^2)-\p_{\xi_2}(b_1b_2)\right],\\
  \text{RHS}_1^{x_1,x_2} = & \text{RHS}_1^{\xi_1,\xi_2} \\
  & + \frac12V^2\left[\p_{\xi_1}(\kappa\p_{\xi_1}\rho)+\p_{\xi_2}(\kappa\p_{\xi_2}\rho)\right]\\
  & + V \big[\p_{\xi_1}(v_1\kappa\p_{\xi_1}\rho)+\p_{\xi_2}(v_1\kappa\p_{\xi_2}\rho)\\
    & \quad\quad+\p_{\xi_1}(\mu\rho\p_{\xi_1}v_1)+\frac12\p_{\xi_2}(\mu\rho\p_{\xi_2}v_1+\p_{\xi_1}v_2)\big],\\
  \text{RHS}_2^{x_1,x_2} = & \text{RHS}_2^{\xi_1,\xi_2} - Vb_1(\p_{x_1}b_1+\p_{x_2}b_2).
\end{array}
\]
The terms associated with $\frac12 V^2$ are cancelled due to the mass equation. The terms associated with $V$ are canceled due to the momentum equation. Therefore, the energy equation is also Galilean invariant.

Finally, we consider the magnetic equation,
\[
\p_t\begin{pmatrix}
b_1\\
b_2
\end{pmatrix}
+
\p_{x_1}
\begin{pmatrix}
  0\\
  u_1b_2-u_2b_1
\end{pmatrix}
+
\p_{x_2}
\begin{pmatrix}
  u_2b_1-u_1b_2\\
  0
\end{pmatrix}
=
\]
\[
\p_{x_1}\begin{pmatrix}
  0\\
  \eta(\p_{x_1}b_2-\p_{x_2}b_1)
\end{pmatrix}
+\p_{x_2}\begin{pmatrix}
  \eta(\p_{x_2}b_1-\p_{x_1}b_2)\\
  0
\end{pmatrix}
+\begin{pmatrix}
-u_1(\p_{x_1}b_1+\p_{x_2}b_2)\\
-u_2(\p_{x_1}b_1+\p_{x_2}b_2)
\end{pmatrix}.
\]
Let us use $\text{LHS}^{x_1,x_2}, \text{RHS}^{x_1,x_2}$ as common notations to respectively denote the left hand side and the right hand side as functions of the spatial variables $x_1,x_2$. The first magnetic equation can be written as
\[
\begin{array}{rl}
\text{LHS}^{x_1,x_2} = & \p_{\tau}b_1- V\p_{\xi_1}b_1+\p_{\xi_2}(v_2b_1-(v_1+V)b_2)\\
  = & \text{LHS}^{\xi_1,\xi_2} - V\p_{\xi_1}b_1-V\p_{\xi_2}b_2,\\
\text{RHS}^{x_1,x_2} = & \eta(\p_{\xi_2}b_1-\p_{\xi_1}b_2)-(v_1+V)(\p_{\xi_1}b_1+\p_{\xi_2}b_2)\\
  = & \text{RHS}^{\xi_1,\xi_2} - V\p_{\xi_1}b_1-V\p_{\xi_2}b_2.
\end{array}
\]
The second magnetic equation can be written as
\[
\begin{array}{rl}
\text{LHS}^{x_1,x_2} = & \p_{\tau}b_2- V\p_{\xi_1}b_2+\p_{\xi_1}((v_1+V)b_2-v_2b_1)\\
  = & \text{LHS}^{\xi_1,\xi_2}-V\p_{\xi_1}b_2+V\p_{\xi_1}b_2,\\
\text{RHS}^{x_1,x_2} = & \eta(\p_{\xi_2}b_1-\p_{\xi_1}b_2)-v_2(\p_{\xi_1}b_1+\p_{\xi_2}b_2)\\
  = & \text{RHS}^{\xi_1,\xi_2}.
\end{array}
\]
Therefore, the magnetic equation is Galilean invariant. It can be seen that without the Powell term, the system would not be Galilean invariant. This was mentioned earlier by \cite{Powell_et_al_1999}.
\subsection{Rotational invariance}\label{appendix:rotational_invariance} Consider the same equations in a new system of coordinates $\bxi=(\xi_1,\xi_2,\xi_3)^\top$ by the rotation $\bxi=\bR\bx$. The observation frame is rotated through an angle $\theta_y$ around the $x_2-$axis by the rotation matrix $\bR_\psi$ and through an angle $\theta$ around the $x_3-$axis by $\bR_\theta$,
\[
\bR_\psi = 
\begin{pmatrix}
\cos\psi & 0 & \sin\psi\\
0 & 1 & 0\\
-\sin\psi & 0 & \cos\psi
\end{pmatrix}, \;
\bR_\theta = 
\begin{pmatrix}
\cos\theta & \sin\theta & 0\\
-\sin\theta & \cos\theta & 0\\
0 & 0 & 1
\end{pmatrix}.
\]
The full rotation matrix $\bR$ is
\[
\bR\coloneqq\bR(\psi,\theta)=\bR_\psi\bR_\theta=\begin{pmatrix}
\cos\psi\cos\theta & \cos\psi\sin\theta & \sin\psi\\
-\sin\theta & \cos\theta & 0\\
-\sin\psi\cos\theta & -\sin\psi\sin\theta & \cos\psi
\end{pmatrix}.
\]
It is clear that $\bR$ is an orthogonal matrix $\bR^{-1}=\bR^\top$. For a given state, the conserved variable vector $\bsfU_0=(\rho_0,\bbm_0,E_0,\bB_0)^\top$ is rotated to $\bT\bsfU_0$ by the transformation matrix
\[
\bT = \begin{pmatrix}
  1 & \bzero^\top & 0 & \bzero^\top\\
  \bzero & \bR & \bzero & \polO\\
  0 & \bzero^\top & 1 & \bzero^\top\\
  \bzero & \polO & \bzero & \bR
\end{pmatrix},
\]
where $\bzero$ is the zero vector of size 3 and $\polO$ is the $3\times3$ zero matrix.
The relation of the derivatives is
\[
\begin{pmatrix}\p_{\xi_1}\\\p_{\xi_2}\\\p_{\xi_3}\end{pmatrix}
=\bR
\begin{pmatrix}\p_{x_1}\\\p_{x_2}\\\p_{x_3}\end{pmatrix}
\text{ or }
\begin{pmatrix}\p_{x_1}\\\p_{x_2}\\\p_{x_3}\end{pmatrix}
=\bR^\top
\begin{pmatrix}\p_{\xi_1}\\\p_{\xi_2}\\\p_{\xi_3}\end{pmatrix}.
\]
Rotational invariance of \eqref{eq:mhd1} means that \eqref{eq:mhd1} does not change under the above rotation,
\begin{equation}\label{eq:mhd_rotational_invariance}
\p_t(\bT\bsfU)+\p_{\xi_1}\bsfF_1(\bT\bsfU)+\p_{\xi_2}\bsfF_2(\bT\bsfU)+\p_{\xi_3}\bsfF_3(\bT\bsfU)=0,
\end{equation}
where $\bsfF(\bsfU) \coloneqq \bsfF_{\calE}(\bsfU)+\bsfF_{\calB}(\bsfU)$, and $\bsfF_1(\bsfU), \bsfF_2(\bsfU), \bsfF_3(\bsfU)$ are $x_1-$, $x_2-$, $x_3-$ components of $\bsfF$.
Because $\p_t(\bT\bsfU) = \bT\p_t(\bsfU)$, substituting \eqref{eq:mhd1} for $\p_t(\bsfU)$ in \eqref{eq:mhd_rotational_invariance} gives
\[
\bT\left(\bR^\top\begin{pmatrix}\p_{\xi_1}\\\p_{\xi_2}\\\p_{\xi_3}\end{pmatrix}\right)\SCAL\begin{pmatrix}\bsfF_1(\bsfU)\\\bsfF_2(\bsfU)\\\bsfF_3(\bsfU)\end{pmatrix} = \p_{\xi_1}\bsfF_1(\bT\bsfU)+\p_{\xi_2}\bsfF_2(\bT\bsfU)+\p_{\xi_3}\bsfF_3(\bT\bsfU).
\]
We now group the terms associating to $\p_{\xi_1}$ to get an equivalent statement to \eqref{eq:mhd_rotational_invariance}. Without loss of generality, the following theorem states rotational invariance of \eqref{eq:mhd1} in the form of \cite[Theorem 1]{Billett_Toro_1998}.
\begin{lemma}\label{lemma:mhd_rotational_invariance}
The nonlinear advective MHD flux $\bsfF(\bsfU)$ satisfies
\[
\bR_{11}\bsfF_1(\bsfU)+\bR_{12}\bsfF_2(\bsfU)+\bR_{13}\bsfF_3(\bsfU)=\bT^{-1}\bsfF_1(\bT\bsfU).
\]
\end{lemma}

Direct substitution can be used to verify the equality in Lemma~\ref{lemma:mhd_rotational_invariance}. We now consider the case where Lemma~\ref{lemma:mhd_rotational_invariance} holds and there is an additional source term $\Psi$. Assume that $\Psi = \Psi^{\p_{x_1}}+\Psi^{\p_{x_2}}+\Psi^{\p_{x_3}}$ where $\Psi^{\p_{x_i}}$ contains only quasilinear terms of first derivatives with respect to $x_i$, the condition of rotational invariance on $\Psi$ is
\begin{equation}\label{eq:rotation_Powell_test}
\Psi^{\p_{x_1}}(\bsfU) = \bT^{-1}\left(\bR_{11}\Psi^{\p_{x_1}}(\bT\bsfU)+\bR_{21}\Psi^{\p_{x_2}}(\bT\bsfU)+\bR_{31}\Psi^{\p_{x_3}}(\bT\bsfU)\right).
\end{equation}
The equality \eqref{eq:rotation_Powell_test} is true if the source term $\Psi$ is chosen to be one of the divergence source terms $\Psi(\alpha_{\bbm},\alpha_E,\alpha_{\bB})$ for any $\alpha_{\bbm},\alpha_E,\alpha_{\bB}$. Thus, the Powell term $\Psi_{\mathrm{Powell}}$ and the Janhunen term $\Psi_{\mathrm{Janhunen}}$ are rotational invariant. Now we consider the regularized equation \eqref{eq:mhd_viscous}. The rotational invariance assumption gives an extra condition on the viscous flux,
\begin{equation}\label{eq:GP_rotational_invariance}
\bT\begin{pmatrix}\p_{x_1}\\\p_{x_2}\\\p_{x_3}\end{pmatrix}\SCAL\;\bsfF_{\calV}(\bsfU) = \left(\bR\begin{pmatrix}\p_{x_1}\\\p_{x_2}\\\p_{x_3}\end{pmatrix}\right)\SCAL\;\bsfF_{\calV,\xi}(\bT\bsfU),
\end{equation}
where all the partial derivatives in $\bsfF_{\calV,\xi}$ are with respect to $\xi$. Denote $\bsfF_{\calV,1},\bsfF_{\calV,2},\bsfF_{\calV,3}$ the three $x_i-$ components of $\bsfF_{\calV}$. The viscous flux $\bsfF_{\calV}$ is a quasilinear function of first order spatial derivatives, $\bsfF_{\calV,j}=\sum_{i=1}^3\bsfF_{\calV,j}^{\p_{x_i}}$ where $\bsfF_{\calV,j}^{\p_{x_i}}$ contains all the $\p_{x_i}$ terms of $\bsfF_{\calV,j}$, for $j = 1,2,3$.

The following proposition simplifies \eqref{eq:GP_rotational_invariance} and shows a slightly more convenient way to check if a viscous regularization to \eqref{eq:mhd_viscous} preserves rotational invariance.

\begin{proposition}\label{proposition:rotational_invariance}
A viscous flux $\bsfF_\calV$ in \eqref{eq:mhd_viscous} is rotationally invariant if the following statements are true:
\[
\begin{array}{rl}
  (i) & \bT\bsfF_{\calV,1}^{\p_{x_1}}(\bsfU)=\bR_{11}\bsfF_{\calV,1}^{\p_{x_1}}(\bT\bsfU)+\bR_{21}\bsfF_{\calV,2}^{\p_{x_1}}(\bT\bsfU)+\bR_{31}\bsfF_{\calV,3}^{\p_{x_1}}(\bT\bsfU);\\
  (ii) & \begin{array}{r}\bT\left(\p_{x_1}\bsfF_{\calV,1}^{\p_{x_2}}(\bsfU)+\p_{x_2}\bsfF_{\calV,1}^{\p_{x_1}}(\bsfU)\right)=\\ \\ \end{array}
  \begin{array}{l}
    \bR_{11}\p_{x_1}\bsfF_{\calV,1}^{\p_{x_2}}(\bT\bsfU)+\bR_{12}\p_{x_2}\bsfF_{\calV,1}^{\p_{x_1}}(\bT\bsfU)\\
    +\bR_{21}\p_{x_1}\bsfF_{\calV,2}^{\p_{x_2}}(\bT\bsfU)+\bR_{22}\p_{x_2}\bsfF_{\calV,2}^{\p_{x_1}}(\bT\bsfU)\\
    +\bR_{31}\p_{x_1}\bsfF_{\calV,3}^{\p_{x_2}}(\bT\bsfU)+\bR_{32}\p_{x_2}\bsfF_{\calV,3}^{\p_{x_1}}(\bT\bsfU).
  \end{array}
\end{array}
\]
\end{proposition}

The first condition is derived by matching the $\p_{x_1}\p_{x_1}$ terms in \eqref{eq:GP_rotational_invariance}. The second condition matches the $\p_{x_1}\p_{x_2}$ terms. Since the rotation is arbitrary, \eqref{eq:GP_rotational_invariance} follows if $(i)$ and $(ii)$ hold for all $\psi,\theta$.

\begin{lemma}\label{lemma:GP_rotational_invariance}
The GP viscous flux $\bsfF_\calV^{\mathrm{GP}}$ given by \eqref{eq:GP_mhd_flux} is rotationally invariant.
\end{lemma}

It can be verified that the equalities in \cref{proposition:rotational_invariance} hold for the GP viscous flux $\bsfF_\calV^{\mathrm{GP}}$. A symbolic Matlab script to verify \cref{lemma:GP_rotational_invariance} is included in \cref{sec:matlab_code}. \cref{thm:rotational_invariance} is followed by \cref{lemma:mhd_rotational_invariance} and \cref{lemma:GP_rotational_invariance}.

Similarly, we can show that the MHD system \eqref{eq:mhd_viscous} regularized by the monolithic parabolic flux $\epsilon\GRAD\bsfU$ is Galilean and rotationally invariant.

\subsubsection{Matlab code to check rotational invariance}\label{sec:matlab_code}
\begin{small}
\begin{verbatim}
syms x y z real
syms theta_y theta_z real
syms rho(x,y,z) u(x,y,z) v(x,y,z) w(x,y,z) E(x,y,z) p(x,y,z)
syms b1(x,y,z) b2(x,y,z) b3(x,y,z)
syms kappa(x,y,z) mu(x,y,z) nu(x,y,z)
U = [rho; rho*u; rho*v; rho*w; E; b1; b2; b3];
R_y = [ cos(theta_y)	0	sin(theta_y);...
        0               1	0;...
        -sin(theta_y)	0	cos(theta_y)];
R_z = [ cos(theta_z)    sin(theta_z)	0 ;...
        -sin(theta_z)   cos(theta_z)	0 ;...
        0               0               1];  
R = R_y*R_z;
T = sym(eye(8));
T(2:4,2:4) = R;
T(6:8,6:8) = R;
U_xyz = U(x,y,z);
TU = T*U;
TU_xyz = TU(x,y,z);
I = eye(3);

% matching \p_x'x'
LHS = Fvisc_u(I,1,1,U_xyz);
RHS = R(1,1)*Fvisc_u(R,1,1,TU_xyz)+R(2,1)*Fvisc_u(R,2,1,TU_xyz)...
     +R(3,1)*Fvisc_u(R,3,1,TU_xyz);
compare = simplify(T*LHS-RHS) % should be zeros
bb = Fvisc_u(I,1,1,U_xyz)+Fvisc_u(I,1,2,U_xyz)+Fvisc_u(I,1,3,U_xyz);

% matching \p_x'y'
LHS = diff(Fvisc_u(I,1,2,U_xyz),x)+diff(Fvisc_u(I,2,1,U_xyz),y);
RHS = R(1,1)*diff(Fvisc_u(R,1,2,TU_xyz),x)...
        +R(1,2)*diff(Fvisc_u(R,1,1,TU_xyz),y)...
     +R(2,1)*diff(Fvisc_u(R,2,2,TU_xyz),x)...
        +R(2,2)*diff(Fvisc_u(R,2,1,TU_xyz),y)...
     +R(3,1)*diff(Fvisc_u(R,3,2,TU_xyz),x)...
        +R(3,2)*diff(Fvisc_u(R,3,1,TU_xyz),y);
compare = simplify(T*LHS-RHS) % should be zeros

% Test the following viscous flux
function Fvisc = Fvisc_u(R,i,j,uu)
    % i: i = 1 => F, i = 2 => G, i = 3 => H
    % as in Billett & Toro, 1998, Eq. (11)
    % j: which derivative (in the original coordinates) you want to collect?
    syms x y z real
    syms kappa(x,y,z)
    assumeAlso(kappa(x,y,z),'real');
    rho = uu(1); u = uu(2)/uu(1); v = uu(3)/uu(1); w = uu(4)/uu(1); E = uu(5);
    b1 = uu(6); b2 = uu(7); b3 = uu(8);
    u_ = uu(i+1)/uu(1);
    R_ij = sym(zeros(1,3));
    R_1j = sym(zeros(1,3));R_2j = sym(zeros(1,3));R_3j = sym(zeros(1,3));
    R_ij(j) = R(i,j);R_1j(j) = R(1,j);R_2j(j) = R(2,j);R_3j(j) = R(3,j);
    Fvisc = [R_ij*gradient(rho,[x,y,z]);...
            rho*(R_ij*gradient(u,[x,y,z])+R_1j*gradient(u_,[x,y,z]))/2...
                +R_ij*gradient(rho,[x,y,z])*u;...
            rho*(R_ij*gradient(v,[x,y,z])+R_2j*gradient(u_,[x,y,z]))/2...
                +R_ij*gradient(rho,[x,y,z])*v;...
            rho*(R_ij*gradient(w,[x,y,z])+R_3j*gradient(u_,[x,y,z]))/2...
                +R_ij*gradient(rho,[x,y,z])*w;...
            R_ij*gradient(E-rho*(u^2+v^2+w^2)/2-(b1^2+b2^2+b3^2)/2,[x,y,z])...
              +(u^2+v^2+w^2)/2*R_ij*gradient(rho,[x,y,z])...
              +rho*u*(R_ij*gradient(u,[x,y,z])+R_1j*gradient(u_,[x,y,z]))/2...
              +rho*v*(R_ij*gradient(v,[x,y,z])+R_2j*gradient(u_,[x,y,z]))/2...
              +rho*w*(R_ij*gradient(w,[x,y,z])+R_3j*gradient(u_,[x,y,z]))/2...
              +b1*(R_ij*gradient(b1,[x,y,z])-R_1j*gradient(b_,[x,y,z]))...
              +b2*(R_ij*gradient(b2,[x,y,z])-R_2j*gradient(b_,[x,y,z]))...
              +b3*(R_ij*gradient(b3,[x,y,z])-R_3j*gradient(b_,[x,y,z]));...
            R_ij*gradient(b1,[x,y,z])-R_1j*gradient(b_,[x,y,z]);...
            R_ij*gradient(b2,[x,y,z])-R_2j*gradient(b_,[x,y,z]);...
            R_ij*gradient(b3,[x,y,z])-R_3j*gradient(b_,[x,y,z])];
end
\end{verbatim}
\end{small}








\section{Magnetic helicity}\label{appendix:magnetic_helicity}
We follow a similar proof as \cite[Sec 2]{Gawlik2022}. 

\begin{proposition}
Magnetic helicity is conserved if $\nu = 0$ and if $\DIV \bB = 0$.
\end{proposition}

\begin{proof}
For $\bB$ to admit a vector potential $\bA$, we require that $\bB$ is pointwise divergence-free since the divergence of curl is zero ($\DIV ( \nabla \CROSS \bA ) = 0$). Thus magnetic helicity only exists if $\DIV \bB = 0$. To show conservation of magnetic helicity, we start with the definition of magnetic helicity

\begin{equation}
\begin{split}
\partial_t \int_\Omega \bA \SCAL \bB \ud \bx =  \partial_t (\bA, \bB) = (\bA_t, \bB) + (\bA, \bB_t) = (\bA_t , \nabla \CROSS \bA) + (\bA, \bB_t)\\= (\nabla \CROSS \bA_t , \bA) + (\bA, \bB_t) 
= ( \bB_t , \bA) + (\bA, \bB_t) \\
= 2 ( \DIV ( \bB \otimes \bu - \bu \otimes \bB ) , \bA  ) + 2 \l( \DIV \l(\nu \l( \GRAD \bB - (\GRAD \bB)^\top \r) \r), \bA \r) \\
= 2 ( \DIV ( \bB \otimes \bu - \bu \otimes \bB ) , \bA  ) + 2 \l( \DIV \l(\nu \l( \GRAD \bB - (\GRAD \bB)^\top \r) \r), \bA \r), \\
\end{split}
\end{equation}
where \eqref{eq:IBP cross} was used. Using the identity $\nabla \CROSS (\bu \CROSS \bB) = \DIV ( \bB \otimes \bu - \bu \otimes \bB )$ yields
\begin{equation}
\begin{split}
\partial_t (\bA, \bB) = 2 ( \DIV ( \bB \otimes \bu - \bu \otimes \bB ) , \bA  ) + 2 \l( \DIV \l(\nu \l( \GRAD \bB - (\GRAD \bB)^\top \r) \r), \bA \r) \\
=  2 (\nabla \CROSS (\bu \CROSS \bB), \bA) + 2 \l( \DIV \l(\nu \l( \GRAD \bB - (\GRAD \bB)^\top \r) \r), \bA \r)\\
=  2 ( \bu \CROSS \bB,  \nabla \CROSS \bA) + 2 \l( \DIV \l(\nu \l( \GRAD \bB - (\GRAD \bB)^\top \r) \r), \bA \r) \\
=  2 ( \bu \CROSS \bB,  \bB ) + 2 \l( \DIV \l(\nu \l( \GRAD \bB - (\GRAD \bB)^\top \r) \r), \bA \r) =   2 \l( \DIV \l(\nu \l( \GRAD \bB - (\GRAD \bB)^\top \r) \r), \bA \r),
\end{split}
\end{equation}
where we used \eqref{eq:IBP cross} and the fact that $\bu \CROSS \bB$ is perpendicular to $\bB$. Thus, magnetic helicity is conserved if $\nu = 0$.
\end{proof}

\begin{remark}
If $\nu$ is constant, the term $ 2 \l( \DIV \l(\nu \l( \GRAD \bB - (\GRAD \bB)^\top \r) \r), \bA \r)$ simpifies to

\begin{equation}
\begin{split}
2 \l( \DIV \l(\nu \l( \GRAD \bB - (\GRAD \bB)^\top \r) \r), \bA \r) = 2 \nu \l(  \DIV \l(\l( \GRAD \bB - (\GRAD \bB)^\top \r) \r), \bA \r) \\
 = -2 \nu  (  \nabla \CROSS (\nabla \CROSS \bB) , \bA) = -2 \nu ( \nabla \CROSS \bB , \nabla \CROSS \bA ) = -2 \nu  (\nabla \CROSS \bB , \bB ) \neq 0,
\end{split}
\end{equation}
where \eqref{eq:IBP cross} and $\GRAD^2 \bB - \GRAD (\DIV \bB) = - \nabla \CROSS (\nabla \CROSS \bB) $ was used.

\end{remark}

\section{Tables of convergence studies}
Table~\ref{table:convergence_vortex_P1} shows convergence rates of the proposed schemes with smooth solutions using $\polP_1, \polP_2, \polP_3$ elements, respectively. Table~\ref{table:convergence_briowu} shows convergence behaviour with nonsmooth solutions.

\begin{table}[h!]
    \centering
    \caption{Smooth vortex problem. Convergence of relative errors in L$^1$ and L$^2$ for velocity $\bu_h$ and magnetic field $\bB_h$ at final time $\widehat t=0.05$ using residual viscosity and different polynomial spaces.}
    \label{table:convergence_vortex_P1}
\vspace{0.in}
$\polP_1$ elements
\\
\vspace{0.1in}

    \begin{adjustbox}{max width=\textwidth}
    \begin{tabular}{c|c|c|c|c|c|c|c|c}
    \hline
    \multirow{2}{*}{\#DOFs} &  \multicolumn{4}{c|}{GLM-GP $\bu_h$} &  \multicolumn{4}{c}{GLM-$\text{GP}^s$ $\bu_h$}  \\ \cline{2-9}
     {}   &       L$^1$   &    Rate   &       L$^2$   &    Rate &       L$^1$   &    Rate   &       L$^2$   &    Rate   \\ \hline
     7442 & 6.08E-04 &     -- &   3.37E-03 &     -- & 6.08E-04 &     -- &   3.37E-03 &     -- \\ 
    29282 & 1.51E-04 &   2.01 &   8.36E-04 &   2.01 & 1.51E-04 &   2.01 &   8.36E-04 &   2.01 \\
   116162 & 3.75E-05 &   2.01 &   2.07E-04 &   2.01 & 3.75E-05 &   2.01 &   2.07E-04 &   2.01 \\
   462722 & 9.33E-06 &   2.01 &   5.16E-05 &   2.01 & 9.33E-06 &   2.01 &   5.16E-05 &   2.01 \\ 
   \hline
   \hline
   \multirow{2}{*}{\#DOFs} &  \multicolumn{4}{c|}{GLM-GP $\bB_h$} &  \multicolumn{4}{c}{GLM-$\text{GP}^s$ $\bB_h$}  \\ \cline{2-9}
     {}   &       L$^1$   &    Rate   &       L$^2$   &    Rate &       L$^1$   &    Rate   &       L$^2$   &    Rate   \\ \hline
     7442 & 2.45E-02 &     -- &   2.73E-02 &     -- & 2.45E-02 &     -- &   2.73E-02 &     -- \\ 
    29282 & 6.05E-03 &   2.04 &   6.75E-03 &   2.04 & 6.05E-03 &   2.04 &   6.75E-03 &   2.04 \\ 
   116162 & 1.50E-03 &   2.03 &   1.67E-03 &   2.03 & 1.50E-03 &   2.03 &   1.67E-03 &   2.03 \\ 
   462722 & 3.72E-04 &   2.01 &   4.15E-04 &   2.02 & 3.72E-04 &   2.01 &   4.15E-04 &   2.02 \\ 
   \hline
    \end{tabular}
    \end{adjustbox}
\\
\vspace{0.2in}
$\polP_2$ elements
\\
\vspace{0.1in}

    \begin{adjustbox}{max width=\textwidth}
    \begin{tabular}{c|c|c|c|c|c|c|c|c}
    \hline
    \multirow{2}{*}{\#DOFs} &  \multicolumn{4}{c|}{GLM-GP $\bu_h$} &  \multicolumn{4}{c}{GLM-$\text{GP}^s$ $\bu_h$}  \\ \cline{2-9}
     {}   &       L$^1$   &    Rate   &       L$^2$   &    Rate &       L$^1$   &    Rate   &       L$^2$   &    Rate   \\ \hline
     7442 & 1.90E-04 &     -- &   1.13E-03 &     -- & 1.90E-04 &     -- &   1.13E-03 &     -- \\ 
    29282 & 3.35E-05 &   2.50 &   1.97E-04 &   2.51 & 3.35E-05 &   2.50 &   1.97E-04 &   2.51 \\
   116162 & 7.81E-06 &   2.10 &   4.57E-05 &   2.11 & 7.81E-06 &   2.10 &   4.57E-05 &   2.11 \\
   462722 & 1.96E-06 &   2.00 &   1.14E-05 &   2.00 & 1.96E-06 &   2.00 &   1.14E-05 &   2.00 \\ 
   \hline
   \hline
   \multirow{2}{*}{\#DOFs} &  \multicolumn{4}{c|}{GLM-GP $\bB_h$} &  \multicolumn{4}{c}{GLM-$\text{GP}^s$ $\bB_h$}  \\ \cline{2-9}
     {}   &       L$^1$   &    Rate   &       L$^2$   &    Rate &       L$^1$   &    Rate   &       L$^2$   &    Rate   \\ \hline
     7442 & 7.64E-03 &     -- &   9.00E-03 &     -- & 7.64E-03 &     -- &   9.00E-03 &     -- \\ 
    29282 & 1.25E-03 &   2.64 &   1.42E-03 &   2.70 & 1.25E-03 &   2.64 &   1.42E-03 &   2.70 \\ 
   116162 & 2.76E-04 &   2.19 &   2.98E-04 &   2.27 & 2.76E-04 &   2.19 &   2.98E-04 &   2.27 \\ 
   462722 & 6.82E-05 &   2.02 &   7.16E-05 &   2.06 & 6.82E-05 &   2.02 &   7.16E-05 &   2.06 \\ \hline
    \end{tabular}
    \end{adjustbox}

\vspace{0.2in}
$\polP_3$ elements
\\
\vspace{0.1in}

   \label{table:convergence_vortex_P3}
    \begin{adjustbox}{max width=\textwidth}
    \begin{tabular}{c|c|c|c|c|c|c|c|c}
    \hline
    \multirow{2}{*}{\#DOFs} &  \multicolumn{4}{c|}{GLM-GP $\bu_h$} &  \multicolumn{4}{c}{GLM-$\text{GP}^s$ $\bu_h$}  \\ \cline{2-9}
     {}   &       L$^1$   &    Rate   &       L$^2$   &    Rate &       L$^1$   &    Rate   &       L$^2$   &    Rate   \\ \hline
     7442 & 1.14E-04 &     -- &   5.80E-04 &     -- & 1.14E-04 &     -- &   5.80E-04 &     -- \\ 
    29282 & 8.09E-06 &   3.81 &   4.90E-05 &   3.56 & 8.09E-06 &   3.81 &   4.90E-05 &   3.56 \\
   116162 & 5.26E-07 &   3.94 &   3.93E-06 &   3.64 & 5.26E-07 &   3.94 &   3.93E-06 &   3.64 \\
   462722 & 4.05E-08 &   3.70 &   3.93E-07 &   3.32 & 4.05E-08 &   3.70 &   3.93E-07 &   3.32 \\ 
   \hline
   \hline
   \multirow{2}{*}{\#DOFs} &  \multicolumn{4}{c|}{GLM-GP $\bB_h$} &  \multicolumn{4}{c}{GLM-$\text{GP}^s$ $\bB_h$}  \\ \cline{2-9}
     {}   &       L$^1$   &    Rate   &       L$^2$   &    Rate &       L$^1$   &    Rate   &       L$^2$   &    Rate   \\ \hline
     7442 & 4.33E-03 &     -- &   4.31E-03 &     -- & 4.33E-03 &     -- &   4.31E-03 &     -- \\ 
    29282 & 2.68E-04 &   4.06 &   2.83E-04 &   3.98 & 2.68E-04 &   4.06 &   2.83E-04 &   3.98 \\ 
   116162 & 1.67E-05 &   4.04 &   2.19E-05 &   3.71 & 1.67E-05 &   4.04 &   2.19E-05 &   3.71 \\ 
   462722 & 1.42E-06 &   3.56 &   2.41E-06 &   3.20 & 1.42E-06 &   3.56 &   2.41E-06 &   3.20 \\ \hline
    \end{tabular}
    \end{adjustbox}
\end{table}

\begin{table}[h!]
    \centering
    \caption{Brio-Wu problem. Convergence of relative errors in L$^1$ and L$^2$ against a fine reference solution. $\polP_1$ elements, for both GP and resistive MHD fluxes using residual viscosity.}
    \label{table:convergence_briowu}
    \begin{adjustbox}{max width=\textwidth}
    \begin{tabular}{c|c|c|c|c|c|c|c|c}
    \hline
    \multirow{2}{*}{\#DOFs} &  \multicolumn{4}{c|}{GP flux $\rho_h$} &  \multicolumn{4}{c}{Resistive MHD flux $\rho_h$}  \\ \cline{2-9}
     {}  &    L$^1$   &   Rate &     L$^2$  &   Rate &   L$^1$  &   Rate &    L$^2$   &  Rate  \\ \hline
     161 &   3.02E-02 &     -- &   6.35E-02 &     -- & 2.91E-02 &     -- &   6.00E-02 &     -- \\ 
     321 &   1.65E-02 &   0.88 &   4.46E-02 &   0.51 & 1.68E-02 &   0.80 &   4.24E-02 &   0.50 \\
     641 &   9.21E-03 &   0.84 &   3.17E-02 &   0.49 & 9.49E-03 &   0.83 &   3.02E-02 &   0.49 \\
    1281 &   4.94E-03 &   0.90 &   2.25E-02 &   0.50 & 5.12E-03 &   0.89 &   2.15E-02 &   0.49 \\ 
    \hline
    \hline
    \multirow{2}{*}{\#DOFs} &  \multicolumn{4}{c|}{GP flux $E_h$} &  \multicolumn{4}{c}{Resistive MHD flux $E_h$}  \\ \cline{2-9}
     {}  &    L$^1$   &   Rate &     L$^2$  &   Rate &   L$^1$  &   Rate &    L$^2$   &  Rate  \\ \hline
     161 &   2.63E-02 &     -- &   5.38E-02 &     -- & 2.78E-02 &     -- &   5.44E-02 &     -- \\ 
     321 &   1.41E-02 &   0.91 &   3.88E-02 &   0.47 & 1.48E-02 &   0.92 &   3.77E-02 &   0.53 \\
     641 &   7.47E-03 &   0.92 &   2.75E-02 &   0.50 & 7.88E-03 &   0.91 &   2.69E-02 &   0.49 \\
    1281 &   4.03E-03 &   0.89 &   1.94E-02 &   0.50 & 4.24E-03 &   0.90 &   1.90E-02 &   0.50 \\ 
    \hline
    \hline
    \multirow{2}{*}{\#DOFs} &  \multicolumn{4}{c|}{GP flux $\bB_h$} &  \multicolumn{4}{c}{Resistive MHD flux $\bB_h$}  \\ \cline{2-9}
     {}  &    L$^1$   &   Rate &     L$^2$  &   Rate &   L$^1$  &   Rate &    L$^2$   &  Rate  \\ \hline
     161 &   1.39E-02 &     -- &   5.69E-02 &     -- & 1.58E-02 &     -- &   6.03E-02 &     -- \\ 
     321 &   7.94E-03 &   0.82 &   4.46E-02 &   0.35 & 9.11E-03 &   0.80 &   4.60E-02 &   0.39 \\
     641 &   4.16E-03 &   0.93 &   3.05E-02 &   0.55 & 4.76E-03 &   0.94 &   3.15E-02 &   0.55 \\
    1281 &   2.18E-03 &   0.93 &   2.14E-02 &   0.51 & 2.51E-03 &   0.93 &   2.20E-02 &   0.52 \\ 
    \hline
    \end{tabular}
    \end{adjustbox}
\end{table}

\section{Brio-Wu problem: GP flux vs resistive MHD flux}\label{sec:briowu_GP_vs_resistive}
\cblue{We know from the literature that the resistive MHD flux is implemented in various MHD codes. The point of this test is to validate that the proposed viscous flux converges to the same solution, and to compare the two viscous fluxes in presence of shocks. A comparison of the GP flux and the resistive MHD flux using $\polP_1$ elements for the Brio-Wu problem is shown in Figure~\ref{fig:briowu_GP_vs_resistive}. In the first-order plot, the GP flux solution appears to be more diffusive due to mass diffusion. In the high-order plot, the resistive MHD flux exhibits slightly more oscillative behaviors. The two fluxes produce different solutions but no concrete statement can be drawn from here. Furthermore, both the fluxes yield the same order of convergence in L$^1$- and L$^2$-norms, as can be seen from Table~\ref{table:convergence_briowu}. Although there exists a contact wave in this test, the inconsistency of the resistive MHD flux is hidden when other waves get involved.}

\begin{figure}[!ht]
  \centering
  \includegraphics[width=0.49\textwidth]{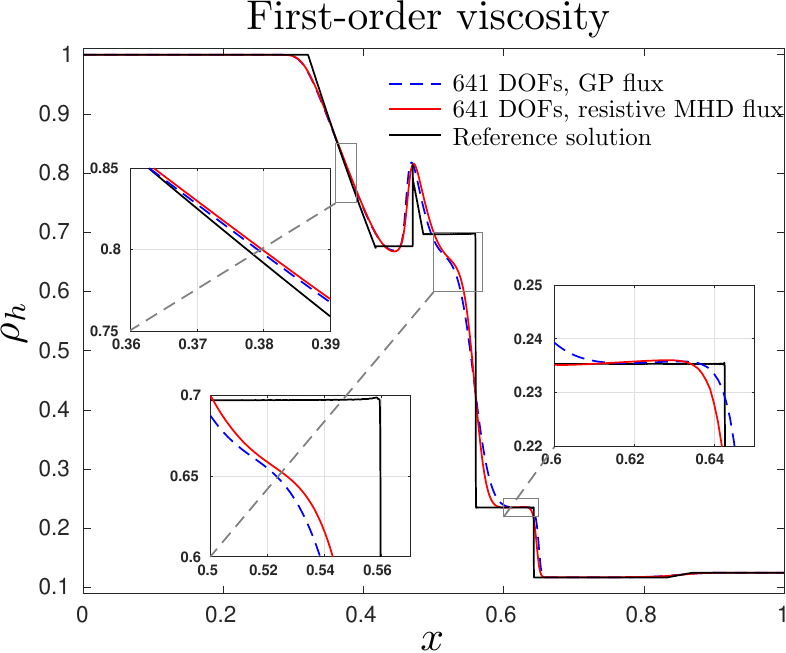}
  \includegraphics[width=0.49\textwidth]{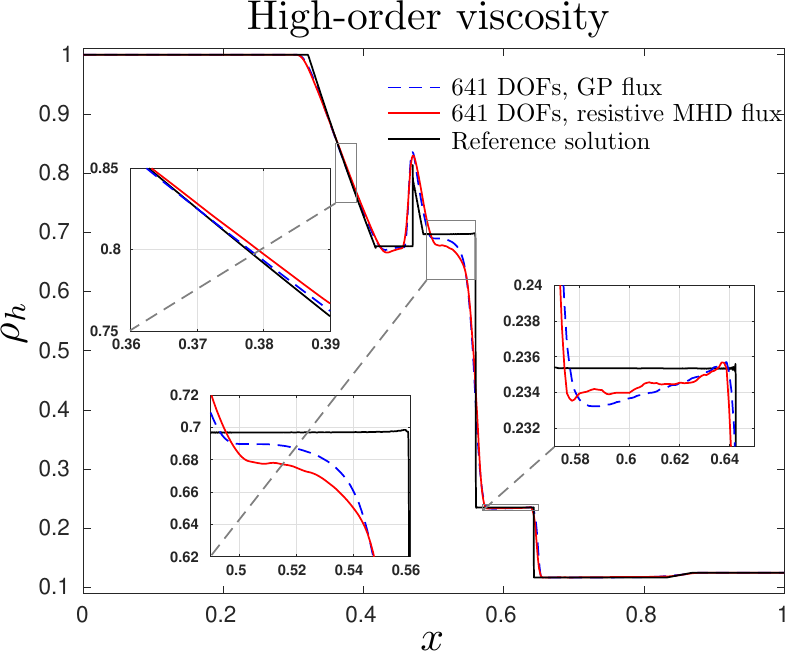}
  \caption{\cblue{A comparison of the GP flux and the resistive MHD flux using $\polP_1$ elements for the Brio-Wu problem: the density profile. Left figure: first-order viscosity is used. Right figure: high-order viscosity using the RV method is used.}}
  \label{fig:briowu_GP_vs_resistive}
\end{figure}

\cblue{A comparison of the GP flux using different polynomial degrees are shown in Figure~\ref{fig:briowu_P1P2P3}. The higher order solutions capture the discontinuities more sharply.}

\begin{figure}[!ht]
  \centering
  \includegraphics[width=0.49\textwidth]{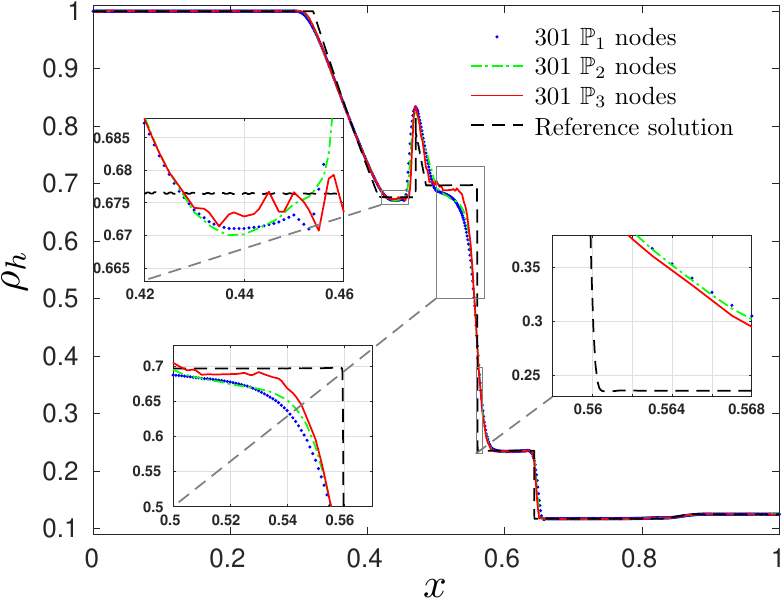}
  \includegraphics[width=0.49\textwidth]{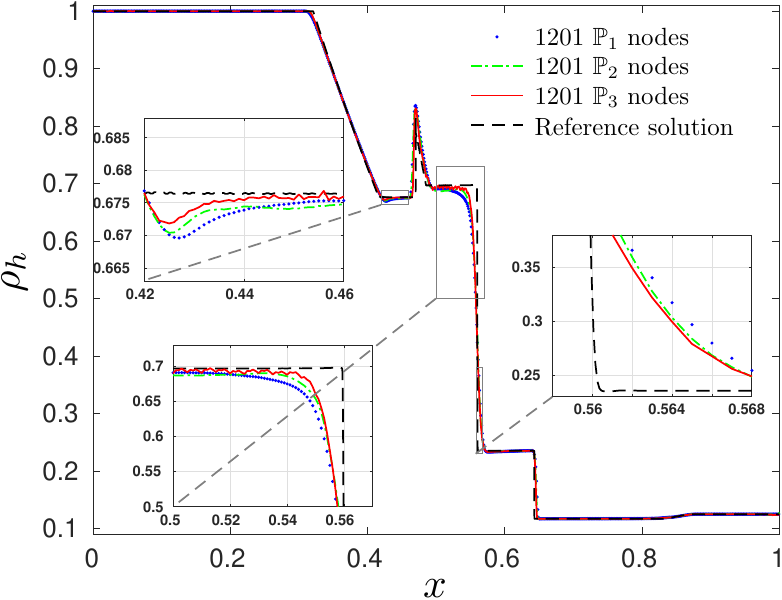}
  \caption{\cblue{A comparison of the GP flux using $\polP_1$, $\polP_2$, $\polP_3$ polynomials on the Brio-Wu problem: the density profile. Left figure: all the solutions are computed using 301 nodes. Right figure: all the solutions are computed using 1201 nodes.}}
  \label{fig:briowu_P1P2P3}
\end{figure}

\section{\cblue{Conservation of angular momentum: a numerical validation}}\label{sec:numerical_results:angular_momentum}

\cblue{A numerical validation of this property is shown in Figure~\ref{fig:angular_momentum}. We perform this test on the smooth vortex problem \cite{Wu_2018b} using 4214 $\polP_1$ nodes until $\hat t = 0.5$. The numerical results are aligned with the continuous analysis in Theorem~\ref{thm:conservation} and the summary reported in Table~\ref{table:conclusion_viscous_model}. No divergence source term is added and the GLM is used for divergence cleaning. We also note that the numerical results agree with those in \cite[Sec 6.3]{lundgren2023fully}. Specifically, the GP flux does not conserve angular momentum while symmetric viscous fluxes in the momentum equations ensure its conservation.}

\begin{figure}[!ht]
  \centering
  \includegraphics[width=0.49\textwidth]{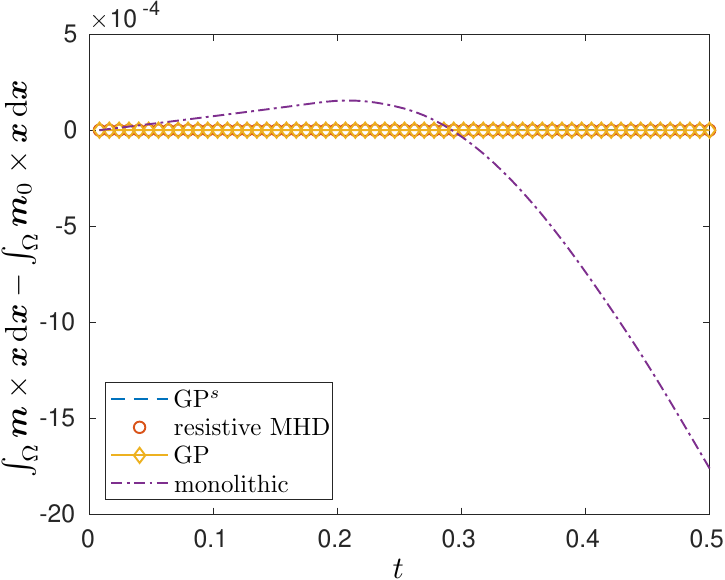}
  \includegraphics[width=0.49\textwidth]{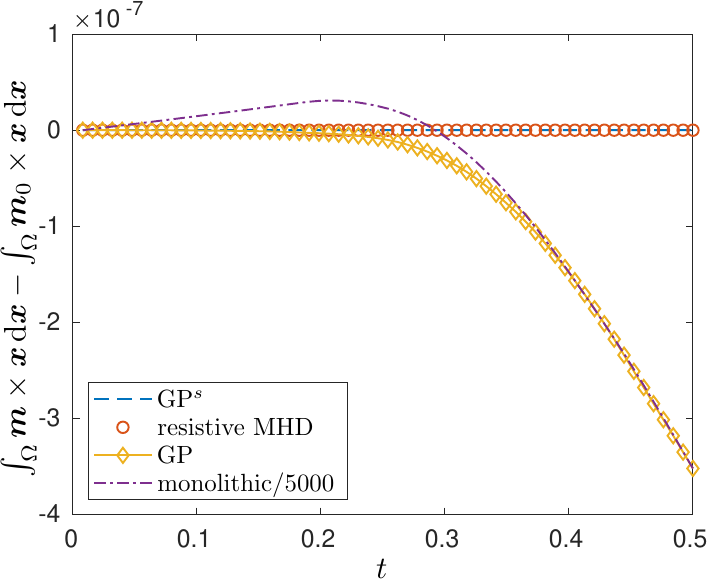}
  \caption{\cblue{Conservation of angular momentum $\int_{\Omega}\bbm\times\bx\ud\bx$ by different viscous fluxes: resistive MHD, GP, GP$^s$, and monolithic \cite{Dao2022a}. Artificial viscosity by the RV method. $\polP_1$ elements. The comparison is separated into two figures because the scales of GP and monolithic lines are different. Right figure: the monolithic line is divided by 5000. The GP$^s$ and the resistive MHD fluxes conserve angular momentum to $\calO(10^{-11})$.}}
  \label{fig:angular_momentum}
\end{figure}

\section{The Orszag-Tang problem}\label{sec:OT}
We include a benchmark by \cite{Orszag_Tang_1979} to compare the GP flux and the GP$^s$ fluxes. The domain is the unit square $\Omega=[0,1]\times[0,1]$. The gas constant is set to $\gamma=\frac{5}{3}$. The initial solution is set as
\begin{align*}
  \rho_0 & = \frac{25}{36\pi},\\
  \bu_0 & = (-\sin(2\pi y),\sin(2\pi x)),\\
  p_0 & = \frac{5}{12\pi},\\
  \bB_0 & = \left(-\frac{\sin(2\pi y)}{\sqrt{4\pi}},\frac{\sin(4\pi x)}{\sqrt{4\pi}}\right).
\end{align*}
Figures~\ref{fig:OT_rho_t05} and \ref{fig:OT_rho_t1} compare the solutions at time $t=0.5$ and $t=1.0$, respectively. It can be seen that the solutions at time $t=0.5$ are visually indistinguishable except for a small difference in the shape of the small upsurge at the middle of the domain. The differences are more visible at time $t=1.0$. We show them in the zoomed-in plots.

\begin{figure}[!ht]
  \centering
     \begin{subfigure}{0.47\textwidth}
         \centering
         \textbf{GLM-GP-MHD}
         \vspace{0.2cm}
     \end{subfigure}
     \hfill
     \begin{subfigure}{0.47\textwidth}
       \centering
       \textbf{GLM-$\text{GP}^s$-MHD}
       \vspace{0.2cm}
     \end{subfigure}
     \hfill
     \begin{subfigure}{0.47\textwidth}
         \centering
         \includegraphics[width=\textwidth]{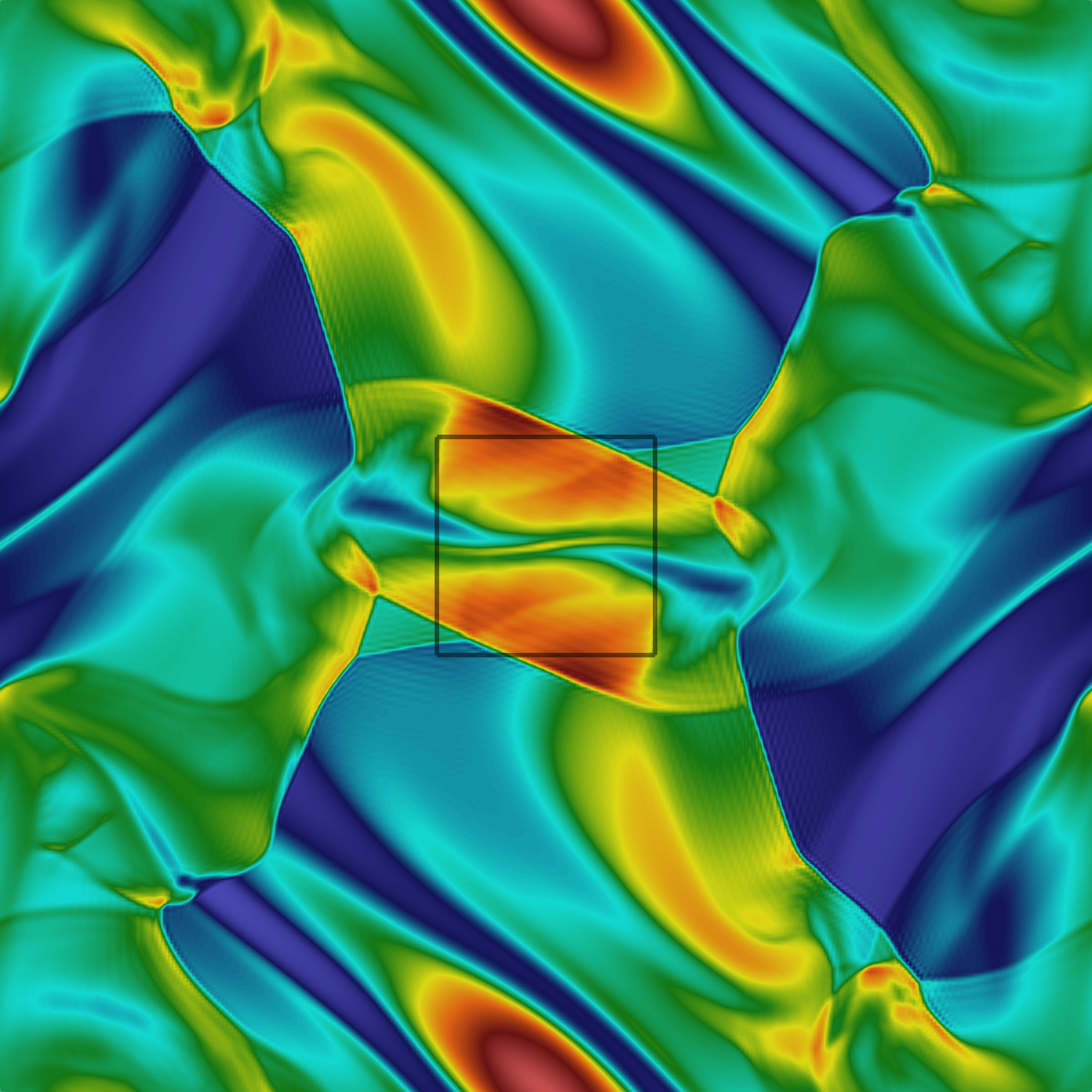}
     \end{subfigure}
     \hfill
     \begin{subfigure}{0.47\textwidth}
         \centering
         \includegraphics[width=\textwidth]{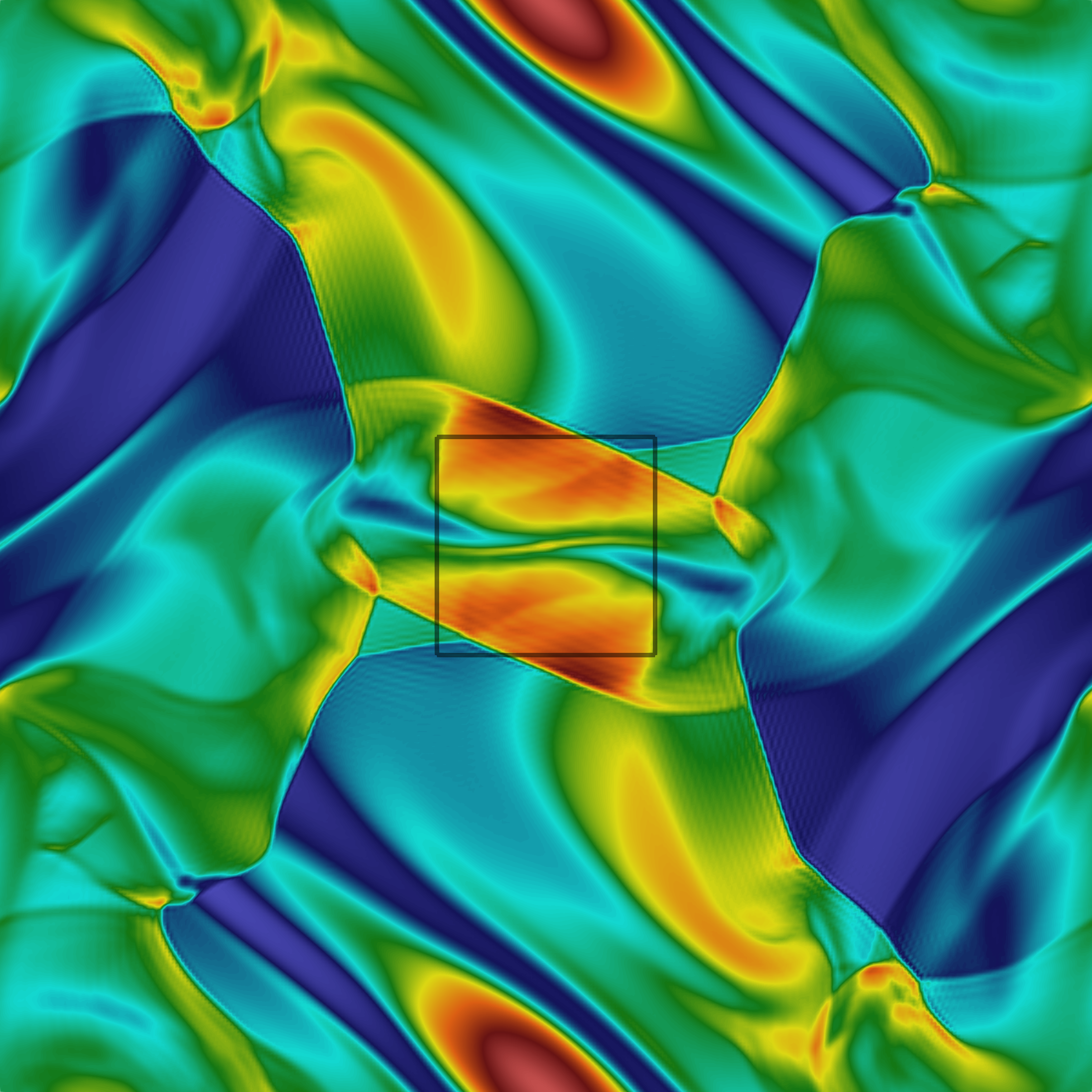}
     \end{subfigure}
     \hfill
     \begin{subfigure}{0.47\textwidth}
         \centering
         \includegraphics[width=\textwidth]{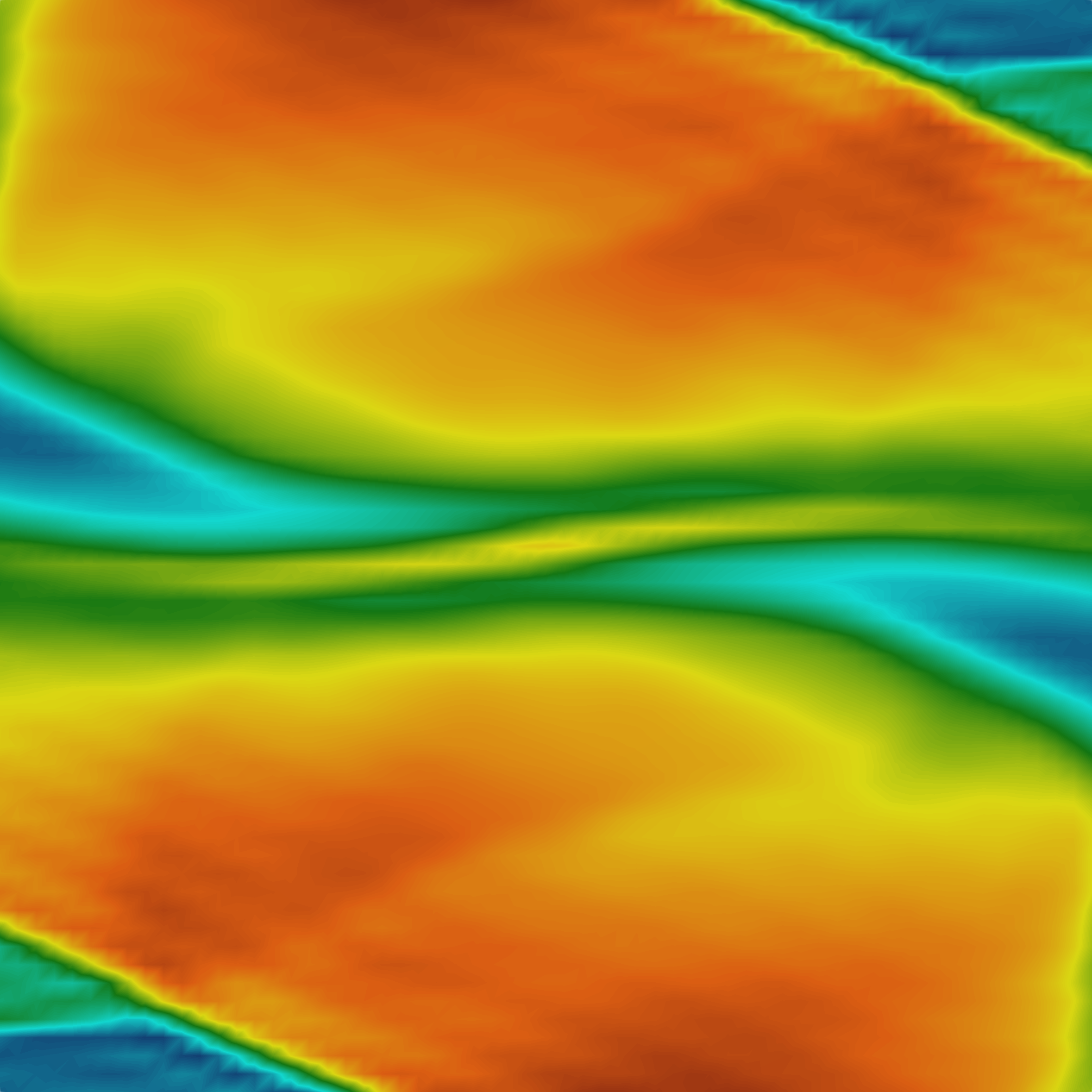}
     \end{subfigure}
     \hfill
     \begin{subfigure}{0.47\textwidth}
         \centering
         \includegraphics[width=\textwidth]{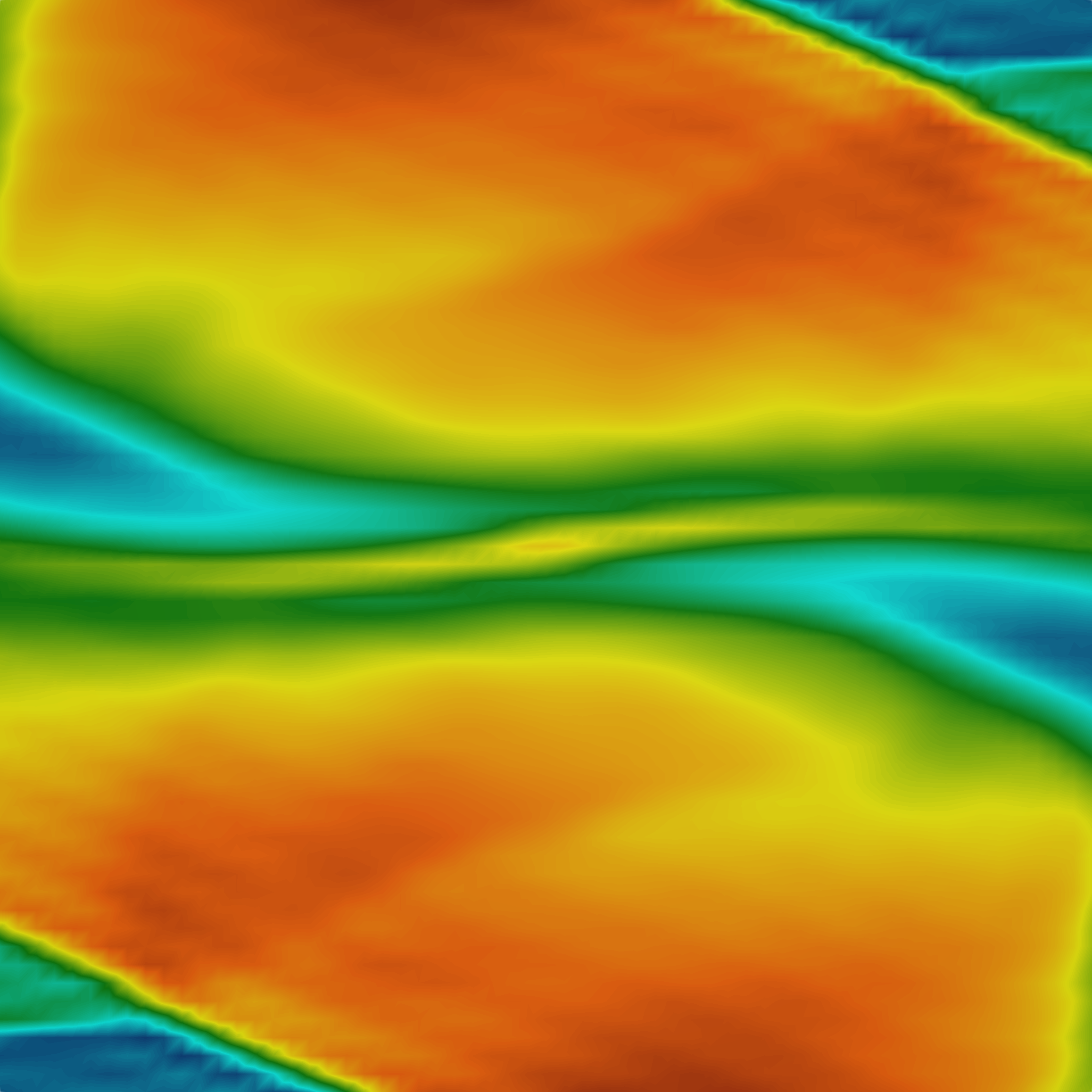}
     \end{subfigure}
     \caption{Density of the Orszag-Tang solution $t=0.5$. Zoomed-in region $[0.4,0.6]\times[0.4,0.6]$}
     \label{fig:OT_rho_t05}
\end{figure}

\begin{figure}[!ht]
  \centering
     \begin{subfigure}{0.47\textwidth}
         \centering
         \textbf{GLM-GP-MHD}
         \vspace{0.2cm}
     \end{subfigure}
     \hfill
     \begin{subfigure}{0.47\textwidth}
       \centering
       \textbf{GLM-$\text{GP}^s$-MHD}
       \vspace{0.2cm}
     \end{subfigure}
     \hfill
     \begin{subfigure}{0.47\textwidth}
         \centering
         \includegraphics[width=\textwidth]{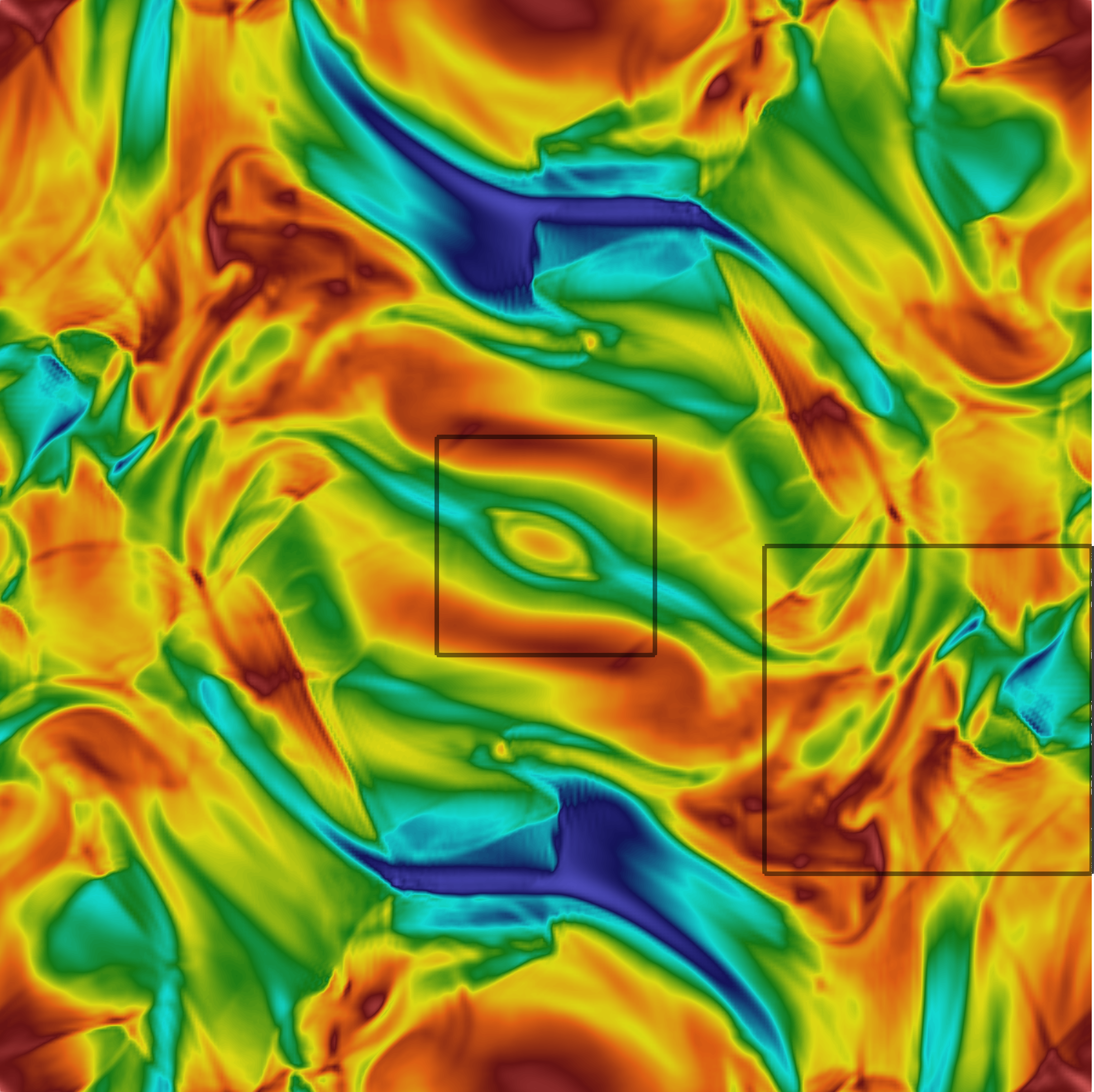}
     \end{subfigure}
     \hfill
     \begin{subfigure}{0.47\textwidth}
         \centering
         \includegraphics[width=\textwidth]{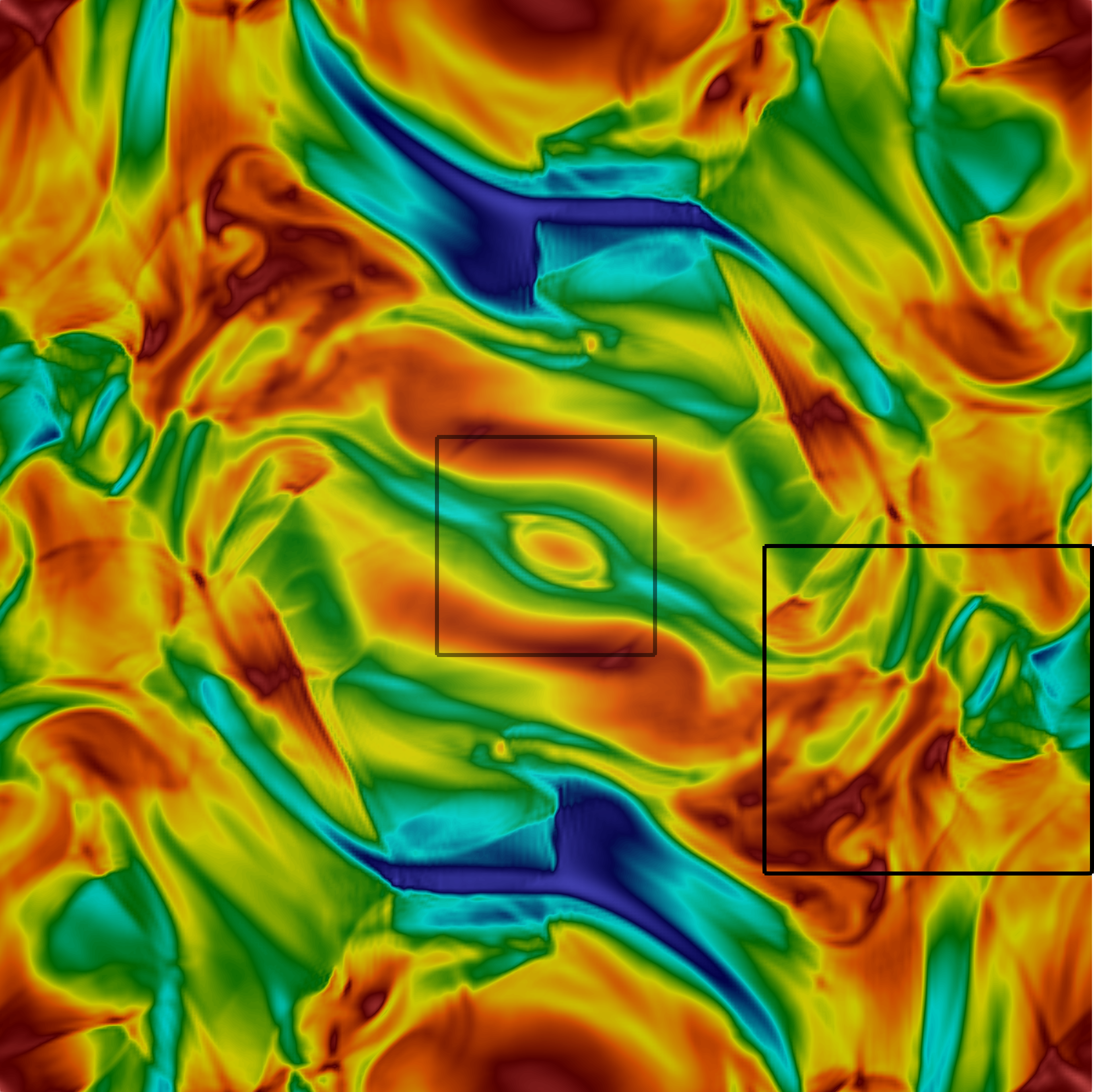}
     \end{subfigure}
     \hfill
     \begin{subfigure}{0.47\textwidth}
         \centering
         \includegraphics[width=\textwidth]{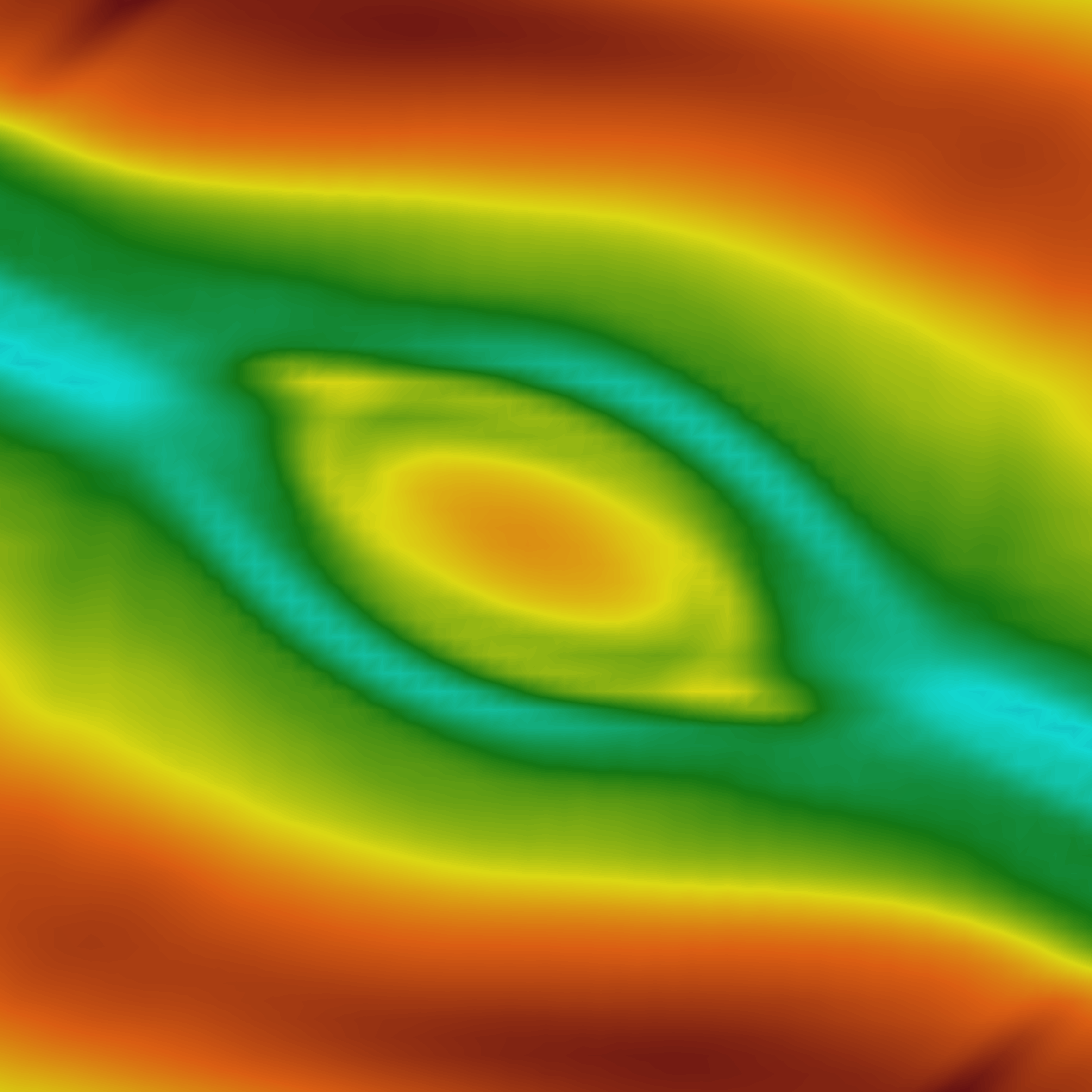}
     \end{subfigure}
     \hfill
     \begin{subfigure}{0.47\textwidth}
         \centering
         \includegraphics[width=\textwidth]{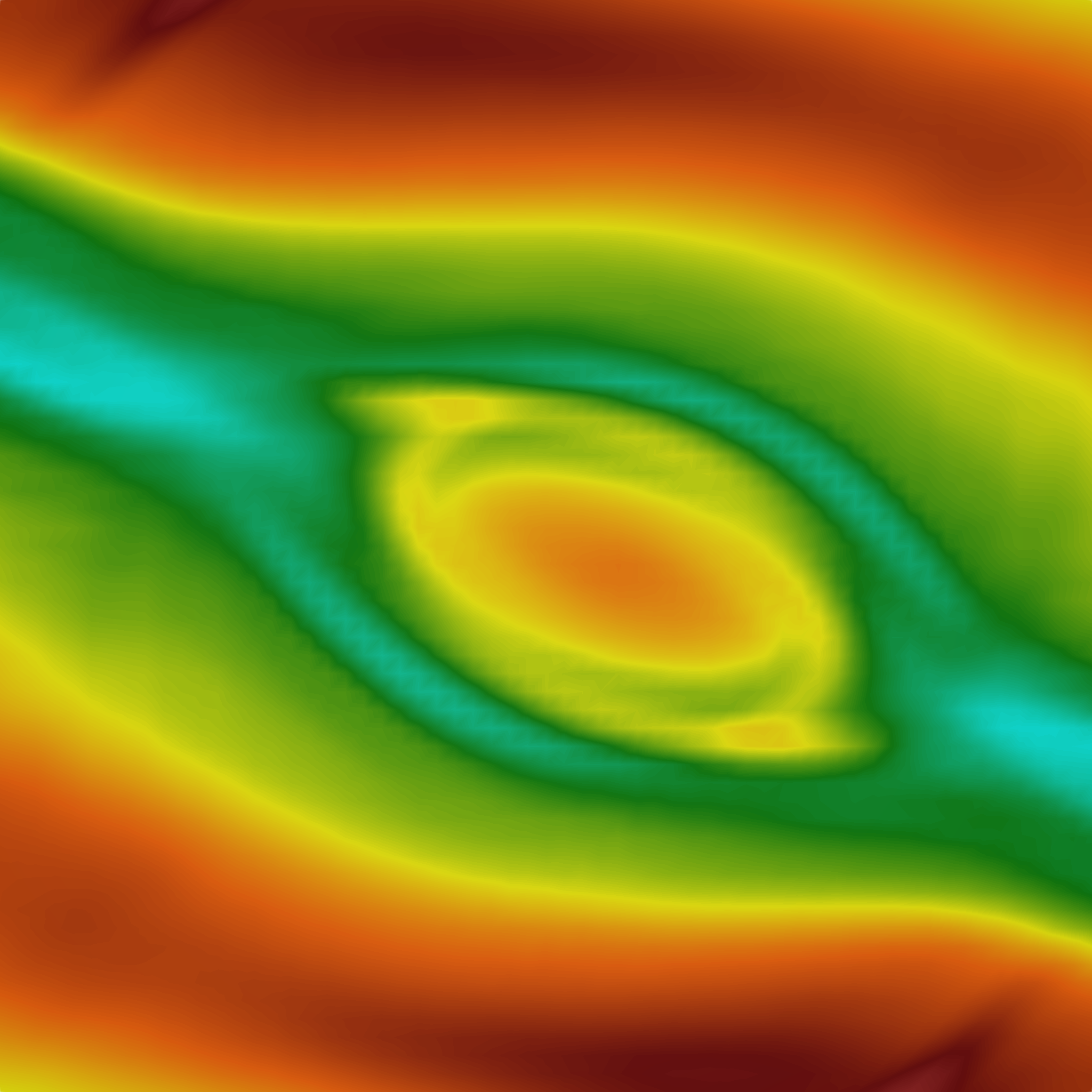}
     \end{subfigure}
     \hfill
     \begin{subfigure}{0.47\textwidth}
         \centering
         \includegraphics[width=\textwidth]{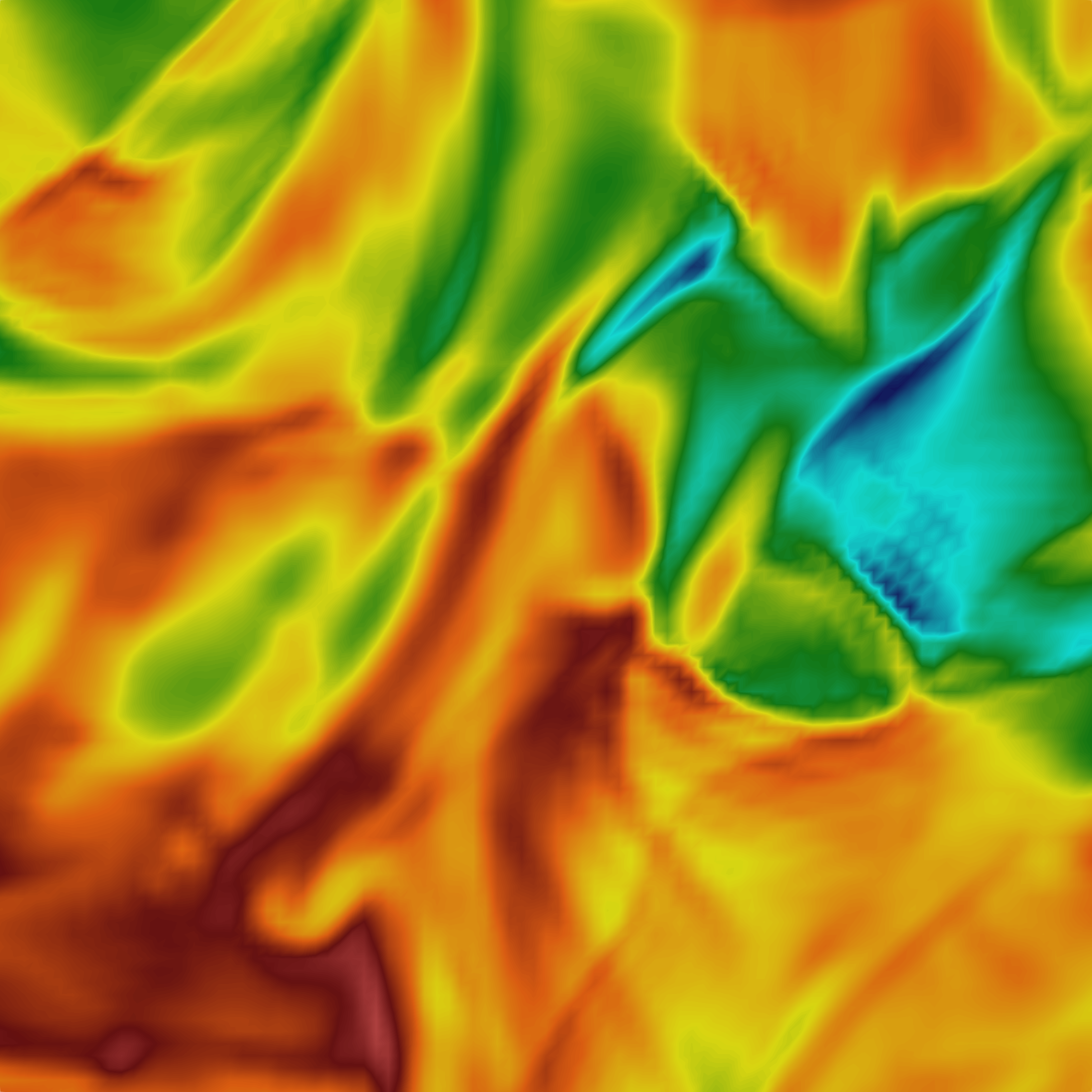}
     \end{subfigure}
     \hfill
     \begin{subfigure}{0.47\textwidth}
         \centering
         \includegraphics[width=\textwidth]{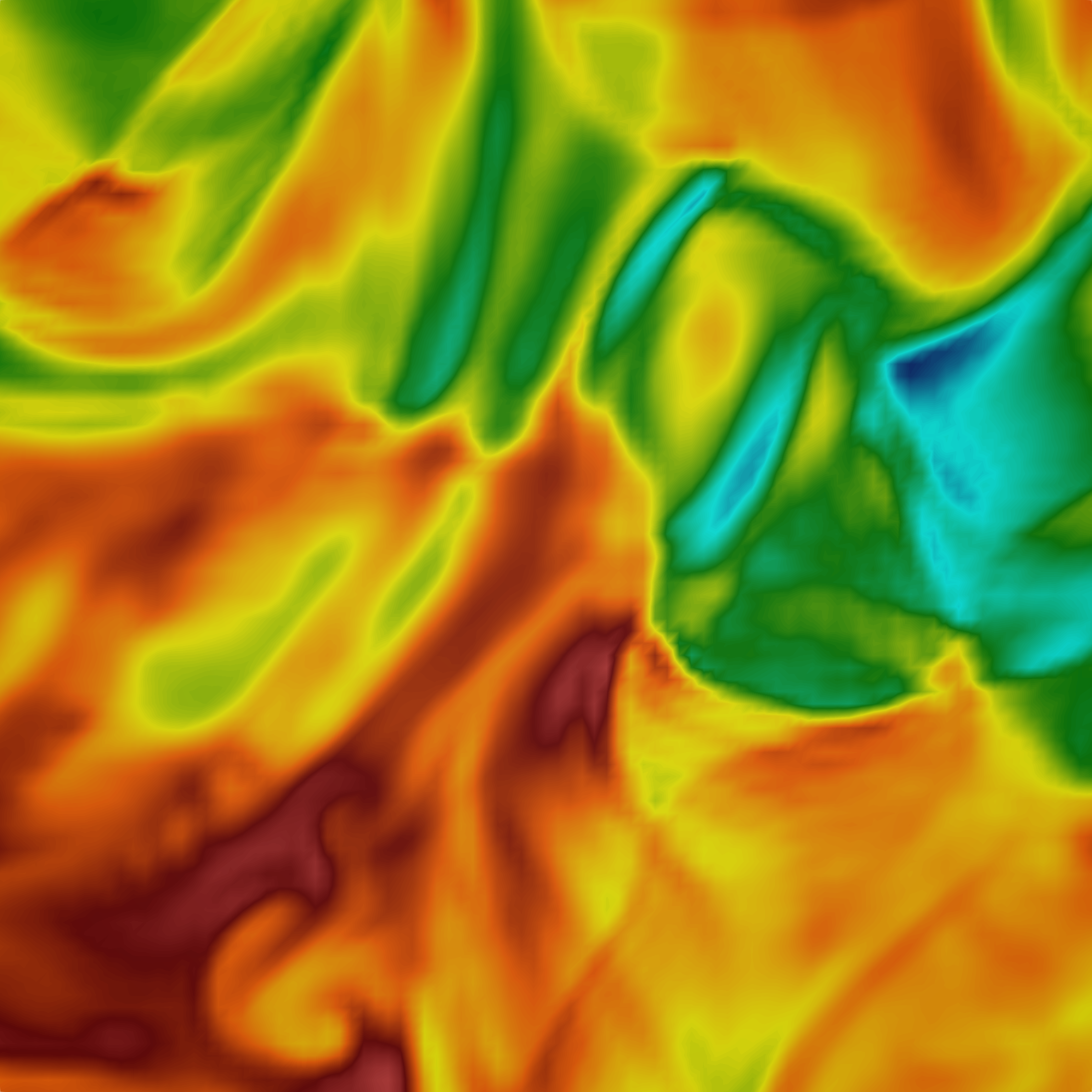}
     \end{subfigure}
     \caption{Density of the Orszag-Tang solution $t=1.0$. Zoomed-in regions $[0.4,0.6]\times[0.4,0.6]$, $[0.7,1.0]\times[0.2,0.5]$.}
     \label{fig:OT_rho_t1}
\end{figure}

\bibliographystyle{siamplain}
\bibliography{ref}
\end{document}

%% file: SISC_shared.tex
\usepackage{lipsum}
\usepackage{amsfonts}
\usepackage{graphicx}
\usepackage{epstopdf}
\usepackage{algorithmic}
\ifpdf
  \DeclareGraphicsExtensions{.eps,.pdf,.png,.jpg}
\else
  \DeclareGraphicsExtensions{.eps}
\fi

\newsiamremark{remark}{Remark}
\newsiamremark{hypothesis}{Hypothesis}
\crefname{hypothesis}{Hypothesis}{Hypotheses}
\newsiamthm{claim}{Claim}

\headers{Viscous regularization for MHD}{T. A. Dao, L. Lundgren, and M. Nazarov}

\title{Viscous regularization of the MHD equations\thanks{Submitted to the editors \today.
\funding{This research is funded by Swedish Research Council (VR) under Grant Number 2021-04620.}}}

\author{Tuan Anh Dao\thanks{Division of Scientific Computing, Department of Information Technology, Uppsala University, Sweden 
  (\email{tuananh.dao@it.uu.se}, \email{lukas.lundgren@math.su.se}, \email{murtazo.nazarov@it.uu.se}).}
\and Lukas Lundgren\footnotemark[2]
\and Murtazo Nazarov\footnotemark[2]}

\usepackage{amsopn}
